\documentclass[12pt]{amsart}
\usepackage{amsfonts}
\usepackage{mathrsfs}
\usepackage{amssymb}
\usepackage{amsbsy}
\usepackage{cite}
\newtheorem{theorem}{Theorem}[section]
\newtheorem{lemma}[theorem]{Lemma}
\newtheorem{corollary}[theorem]{Corollary}

\theoremstyle{definition}
\newtheorem{definition}[theorem]{Definition}

\theoremstyle{remark}
\newtheorem{remark}[theorem]{Remark}

\numberwithin{equation}{section}

\begin{document}

\begin{center}
{\Large \bf  $\boldsymbol{L_p}$ Minkowski problem for electrostatic $\mathfrak{p}$-capacity }
\end{center}

\vskip 20pt

\begin{center}
{{\bf Du~~Zou$^1$\quad\quad\quad\quad Ge~~Xiong$^2$}\\~~ \\1. Department of Mathematics,
Wuhan University of Science and Technology,  Wuhan, \\ 430081, PR China\\
2. School of Mathematical Sciences, Tongji University, Shanghai, 200092, PR China}
\end{center}

\vskip 10pt

\footnotetext{E-mail address: 1. zoudu@wust.edu.cn; 2. xiongge@tongji.edu.cn}
\footnotetext{Research of the authors was supported by NSFC No. 11471206 and NSFC No. 11601399.}

\begin{center}
\begin{minipage}{15cm}
{\bf Abstract}  Existence and uniqueness of the solution to the discrete $L_p$
Minkowski problem for $\mathfrak{p}$-capacity
are proved  when  $p\geq1$ and $1<\mathfrak{p}<n$. For general $L_p$
Minkowski problem for $\mathfrak{p}$-capacity,  existence and uniqueness
of the solution are given when $p\geq1$ and $1<\mathfrak{p}\le 2$.
These results are non-linear
extensions of the very recent solution to the $L_p$ Minkowski problem for $\mathfrak{p}$-capacity when $p=1$
and $1<\mathfrak{p}<n$ by  CNSXYZ,
and the classical solution to the Minkowski problem for electrostatic
capacity when $p=1$ and $\mathfrak{p}=2$ by  Jerison.



\vskip 8pt
\end{minipage}
\end{center}

\vskip 25pt
\section{\bf Introduction}
\vskip 10pt

The setting for this paper is Euclidean $n$-space, $\mathbb{R}^n$. A
\emph{convex body} in $\mathbb{R}^n$ is a compact convex set that has a
non-empty interior. A \emph{polytope} in $\mathbb{R}^n$ is the convex hull
of a finite set of points in $\mathbb{R}^n$ provided it has positive \emph{volume}
(i.e., $n$-dimensional volume).

The Brunn-Minkowski theory (or the theory of mixed volumes) of convex bodies, developed
by Minkowski, Aleksandrov, Fenchel, et al., centers around the study of geometric
functionals of convex bodies as well as the differentials of these functionals.
Usually, the differentials  of these functionals produce \emph{new} geometric measures.
The theory depends heavily on analytic tools such as the cosine transform
 on the unit sphere  ${\mathbb S}^{n-1}$
and Monge-Amp$\grave{e}$re type equations.

A  Minkowski problem is a characterization problem for a geometric measure generated
by convex bodies: It asks for necessary and sufficient conditions in order that
a given measure arises as the measure generated by a convex body. The solution of a
Minkowski problem, in general, amounts to solving a degenerate fully non-linear partial
differential equation. The study of Minkowski problems has a long history and strong
influence on both the Brunn-Minkowski theory and fully non-linear partial differential
equations, see \cite{BookSchneider}.

The classical  Brunn-Minkowski theory begins with the variation of volume functional.

\subsection{Volume, surface area measure and the classical  Minkowski problem}

Without doubt, the most fundamental
geometric functional in the Brunn-Minkowski theory is  volume functional.
It is to see that via the variation of volume functional,  it produces the most
important geometric measure: surface  area measure.

Specifically, if $K$ and $L$ are convex bodies in  $\mathbb{R}^n$, then there exists
a finite Borel measure  $S(K,\cdot)$  on the unit sphere ${\mathbb S}^{n-1}$  known as  the
\emph{surface  area measure} of $K$,  so that
\begin{equation}\label{AleksandrovVariationalFormula}
{\left. {\frac{{dV(K + tL)}}{{dt}}} \right|_{t = {0^ + }}} =
\int\limits_{{\mathbb{S}^{n - 1}}} {{h_L}(\xi )dS(K,\xi )} ,
\end{equation}
where $V$ is the $n-$dimensional volume (i.e., Lebesgue measure in $\mathbb{R}^n$);
the convex body $K + tL=\{x+ty: x\in K, y\in L\}$ is the \emph{Minkowski sum}
of $K$ and $tL$;
$h_L: \mathbb{S}^{n-1}\to\mathbb{R}$ is
the \emph{support function} of $L$, defined by $h_L(\xi)=\max\{\xi\cdot x:
x\in L\}$, with $\xi\cdot x$ denoting the  inner product of
$\xi$ and $x$ in $\mathbb{R}^n$. Formula (\ref{AleksandrovVariationalFormula})
, also called the Aleksandrov variational formula,
suggests that the surface area measure
can be viewed as the differential of volume functional.

The surface area measure $S(K,\cdot)$ of a convex body $K$
can be defined directly, for each Borel set $\omega \subset \mathbb{S}^{n-1}$, by
\begin{equation}\label{SurfaceAreaMeasure}
S(K,\omega)=\mathcal{H}^{n-1}({\rm g}^{-1}_K(\omega)),
\end{equation}
where $\mathcal{H}^{n-1}$ is the $(n-1)$-dimensional Hausdorff measure.
Here the Gauss map
${\rm g}_K: \partial' K\to \mathbb{S}^{n-1}$ is
defined on $\partial'K$ of  those points of $\partial K$
that have a unique outer normal  and  is hence defined $\mathcal{H}^{n-1}$-a.e. on $\partial K$.
The integral in (\ref{AleksandrovVariationalFormula}), divided by the ambient dimension $n$,
is called the \emph{first mixed volume} $V_{1}(K,L)$ of $K$ and $L$, i.e.,
\[  V_{1}(K,L) =\frac{1}{n}\int\limits_{{\mathbb S}^{n-1}}h_L(\xi)dS(K,\xi).\]
It is a generalization of the well-known volume formula
\begin{equation}\label{volumeformula}
V(K)=\frac{1}{n}\int\limits_{{\mathbb S}^{n-1}}h_K(\xi)dS(K,\xi).
\end{equation}

The classical Minkowski problem, which characterizes the surface
area measure, is one of the cornerstones of the Brunn-Minkowski
theory of convex bodies. It reads: \emph{ Given a finite Borel measure $\mu$ on  $\mathbb{S}^{n-1}$,
what are the necessary and
sufficient conditions on $\mu$  so that $\mu$  is the surface area
measure  $S(K,\cdot)$ of a convex body $K$ in $\mathbb{R}^n$?}  More than a century ago,
Minkowski himself \cite{Minkproblem(Minkowski)} solved this problem
for the case when the given measure is either  discrete or has a
continuous density. Aleksandrov \cite{Minkproblem(Aleksandrov)},
\cite{Minkproblem(Aleksandrov2)} and Fenchel-Jessen
\cite{Minkproblem(FenchelJessen)} independently solved the problem
in 1938 for arbitrary measures: If  $\mu$
is not concentrated on any great subsphere of $\mathbb{S}^{n-1}$,
then  $\mu$ is the surface area measure of a convex body if and only
if $\int_{\mathbb{S}^{n-1}}\xi d\mu(\xi)=0$.

Since for strictly convex bodies with smooth boundaries, the
reciprocal of the Gauss curvature  is the density of the surface area measure with
respect to the spherical Lebesgue measure, the Minkowski problem in
differential geometry is to characterize the Gauss curvature of
closed convex hypersurfaces. Analytically, the Minkowski
problem is equivalent to solving a degenerate Monge-Amp$\grave{e}$re equation.
Establishing the regularity of the solution to the
Minkwoski problem is difficult and has led to a long series of
highly influential works, see, e.g.,    Lewy \cite{Minkproblem(Lewy)},
Nirenberg \cite{Minkproblem(Nirenberg)}, Cheng and Yau
\cite{Minkproblem(ChengYau)}, Pogorelov
\cite{Minkproblem(Pogorelov)},  Caffarelli
\cite{Minkproblem(Caffarelli1),Minkproblem(Caffarelli)}.

\subsection{$\boldsymbol L_p$ surface area measure and $\boldsymbol L_p$ Minkowski problem for volume}

The $L_{p}$ Brunn-Minkowski theory is an extension of the classical Brunn-Minkowski theory;
see \cite{LpBM(Firey),LpBM(Ludwig1),LpBM(Ludwig2),LpBM(Lutwak1),LpBM(Lutwak2),LpBM(Lutwak3),LpBM(LYZ1),LpBM(LYZ4),LpBM(LYZ5),LpBM(LYZ6),LpBM(LYZ7),Minkproblem(Zhu2)}.
In 1962, Firey \cite{LpBM(Firey)} introduce  \emph{$L_{p}$ sums} for convex bodies.
Let $1\le p < \infty$. If $K$ and $L$ are convex bodies with the origin in their interiors,
then their \emph{$L_p $ sum} $K +_p L$ is the convex body
defined by
\[ h_{K +_p L}(\xi)^p =h_{K}(\xi)^p +h_{L}(\xi)^p, \quad  \xi\in \mathbb{S}^{n-1}.\]
See also,  \cite{LpBM(Firey), LpBM(Lutwak1),LpBM(GardnerHugWeil), LpBM(LYZ11)}.
Clearly, $K+_1 L=K+L$.

For $t>0$, the \emph{$L_p$ scalar multiplication} $t\cdot_{p}  K$ is the convex body $t^{\frac{1}{p}} K$.

The $L_p$ surface area measure, introduced by Lutwak \cite{LpBM(Lutwak1)}, is a
fundamental notion in the $L_{p}$ theory. For fixed  $p\in\mathbb{R}$, and a
convex body $K$ in $\mathbb{R}^n$ with the origin in its interior, the  \emph{$L_p$ surface area measure}
$S_p(K,\cdot)$ of $K$ is a Borel measure on
$\mathbb{S}^{n-1}$ defined, for  Borel  $\omega \subset \mathbb{S}^{n-1}$, by
\[S_{p}(K, \omega)=\int\limits_{x\in {\rm g}_K^{-1}(\omega)}(x\cdot {\rm g}_K(x))^{1-p}d\mathcal{H}^{n-1}(x).\]
The $L_p$ surface area
measure $S_p(K,\cdot)$ can also be explicitly defined, for  Borel  $\omega \subset \mathbb{S}^{n-1}$,  by
\begin{equation}\label{DefLpSurfaceAreaMeasure}
S_p(K,\omega)=\int\limits_{\omega}h_K(\xi)^{1-p}dS(K,\xi).
\end{equation}
Note that
$S_1(K,\cdot)$ is just the surface area measure $S(K,\cdot)$.
$\frac{1}{n}S_{0}(K, \cdot)$ is the cone-volume measure of convex body $K$,
which is the only ${\rm SL}(n)$ invariant measure among all the
$L_{p}$ surface area measures.  In recent years,  cone-volume
measures have been greatly investigated, e.g.,
\cite{BGMN(Barthe),Gromov,LpBM(Ludwig2),LpBM(Ludwig3), Naor,
LpBM(PaourisWerner),LpBM(Xiong),LpBM(ZouXiong1)}.
$S_{2}(K, \cdot)$ is called the quadratic surface area measure of convex body
$K$, which was studied in \cite{LpBM(Ludwig1)} and \cite{LpBM(LYZ2),
LpBM(LYZ3), LpBM(LYZ10)}.  Applications of the $L_{p}$ surface area
measure to affine isoperimetric inequalities were given in, e.g.,
\cite{LpBM(CampiGronchi),LpBM(LZ), LpBM(LYZ1), LpBM(LYZ6)}.

In \cite{LpBM(Lutwak1)},  Lutwak  established the following $L_{p}$ variational formula for volume
\begin{equation}\label{LutwakVariationalFormula}
{\left. {\frac{{dV(K{ +_p}t \cdot_{p} L)}}{{dt}}} \right|_{t = {0^ +
}}} =\frac{1}{p} \int\limits_{{\mathbb{S}^{n - 1}}} {{h_L}{{(\xi )}^p}d{S_p}(K,\xi )},
\end{equation}
which suggests that the $L_{p}$ surface area measure
can be viewed as the differential of  volume functional of    $L_{p}$ combination of convex bodies.
When $p=1$,
(\ref{LutwakVariationalFormula}) is precisely (\ref{AleksandrovVariationalFormula}).

Lutwak \cite{LpBM(Lutwak1)} initiated the following $L_{p}$ Minkowski problem.

\vskip 10pt \noindent \textbf{$\boldsymbol{L_p}$ Minkowski problem
for volume.} \emph{Suppose $\mu$ is a finite Borel measure  on  $\mathbb{S}^{n-1}$ and $p\in\mathbb{R}$.
What are the necessary and
sufficient conditions on $\mu$  so that $\mu$  is
the $L_p$ surface area measure  $S_p(K,\cdot)$ of a convex body $K$ in $\mathbb{R}^n$?}

\vskip 10pt
$L_{1}$ Minkowski problem is precisely the classical
Minkowski problem. The  $L_{0}$ Minkowski problem, which
characterizes the cone-volume measure, is called the logarithmic
Minkowski problem.  In light of its strong geometric intuition and
fundamental significance, the logarithmic Minkowski problem is
regarded as the  most important case. In 1999, Andrews  \cite{Minkproblem(Andrews1)} proved
Firey's conjecture \cite{Minkproblem(Firey)} that convex surfaces
moving by their Gauss curvature become spherical as they contract to
points. A major breakthrough was made by
B$\rm{\ddot{o}}$r$\rm{\ddot{o}}$czky and LYZ \cite{ LpBM(BLYZ2)} in
2013, who establish the sufficient and necessary conditions for the
existence of a solution to the even logarithmic Minkowski problem.
The $L_{-n}$ Minkowski
problem is the centro-affine Minkowski problem. See  Chou
and Wang \cite{Minkproblem(ChouWang)}, and Zhu
\cite{Minkproblem(Zhu1),Minkproblem(Zhu2)}.

In the  recent ground-breaking  paper  \cite{Minkproblem(HuangLYZ)},
Huang, Lutwak, Yang and Zhang  introduced the dual curvature
measures $\tilde{C}_i(K,\cdot)$, $i=0, 1, \ldots, n,$ of a convex
body $K$ and solved their associated Minkowski problems. These new
geometric measures are precisely the counterparts to the curvature
measures in the dual Brunn-Minkowski theory and open up  a new
passage to the $L_p$ surface area measures, since
$\tilde{C}_n(K,\cdot)$ is just the cone-volume measure of $K$.

By now, the $L_{p}$ Minkowski problem for volume has been
intensively investigated and achieved great developments. See, e.g.,
\cite{Minkproblem(Chen),
Minkproblem(ChouWang),Minkproblem(HaberlLYZ), Minkproblem(HugLYZ),
Minkproblem(Jiang), Minkproblem(LuWang),LpBM(Lutwak1),LpBM(Lutwak3),
LpBM(LYZ5),Minkproblem(Stancu),Minkproblem(Zhu2)}. As applications,
the solutions to $L_{p}$ Minkowski problem for volume have been
used to establish sharp affine isoperimetric inequalities, such as
the affine Moser-Trudinger and the affine Morrey-Sobolev
inequalities, the  affine $L_{p}$ Sobolev-Zhang inequality, etc.
See, e.g.,
\cite{LpBM(BLYZ1),LpBM(CianchiLYZ),LpBM(HaberlSchuster1),LpBM(HaberlSchuster2),LpBM(LYZ4),LpBM(LYZ8),LpBM(Zhang)},
for more details.

\subsection{$\mathfrak{p}$-capacitary measure and Minkowski problem for $\mathfrak{p}$-capacity}

It is worth mentioning that the Minkowski problem for electrostatic $\mathfrak{p}$-capacity is doubtless an
extremely important variant among  Minkowski problems.
Recall that for $1<\mathfrak{p}<n$, the \emph{electrostatic $\mathfrak{p}$-capacity} of a
compact set $K$  in $\mathbb{R}^n$ is defined by
\[ {\rm C}_\mathfrak{p}(K)=\inf\left\{\int\limits_{\mathbb{R}^n}|\nabla
u|^\mathfrak{p}dx: u\in C^\infty_c(\mathbb{R}^n)\; {\rm and}\;
u\ge\chi_K \right\}, \]
where $C^\infty_c(\mathbb{R}^n)$ denotes the
set of functions from $C^\infty(\mathbb{R}^n)$ with compact
supports, and $\chi_K$ is the characteristic function of $K$.
 ${\rm C}_2(K)$ is the classical electrostatic (or
Newtonian) capacity of $K$.
Let $L$ be an arbitrary convex body. Via the variation of  capacity functional ${\rm C}_2(K)$, the classical
Hadamard variational formula
\begin{equation}\label{HadamardVariationalFormula1}
{\left. {\frac{{d{{\rm{C}}_2}(K + tL)}}{{dt}}} \right|_{t
= {0^ + }}} = \int\limits_{{S^{n - 1}}} {{h_L}{{(\xi )}}d{\mu _2}(K,\xi )}
\end{equation}
and its special case, the Poincar$\rm\acute{e}$ capacity formula
\begin{equation}\label{ElectrostaticCapacityFormula}
{ {\rm C}_2(K)=\frac{1}{n-2}\int\limits_{\mathbb{S}^{n-1}}h_{K}(\xi)d\mu_2(K,\xi)}
 \end{equation}
appear. Here, the new measure $\mu_2(K,\cdot)$ is a
finite Borel measure on $\mathbb{S}^{n-1}$,
called the \emph{electrostatic capacitary
measure} of $K$.  Formula  (\ref{HadamardVariationalFormula1})  suggests that the electrostatic capacitary
measure can be viewed as the differential of  capacity functional.

In his celebrated article \cite{Minkproblem(Jerison)}, Jerison
pointed out the resemblance between the  Poincar$\rm\acute{e}$
capacity formula (\ref{ElectrostaticCapacityFormula}) and the volume formula (\ref{volumeformula}) and
also a resemblance between their  variational formulas (\ref{HadamardVariationalFormula1}) and (\ref{AleksandrovVariationalFormula}).
Thus, he initiated
to consider the Minkowski problem for  electrostatic
capacity: \emph{Given a finite Borel measure $\mu$  on
$\mathbb{S}^{n-1}$,  what are the necessary and sufficient conditions on $\mu$ so
that $\mu$ is the electrostatic capacitary measure $\mu_2(K,\cdot)$ of  a convex body $K$ in $\mathbb{R}^n$?}

Jerison \cite{Minkproblem(Jerison)} solved, in full generality, the
Minkowski problem for electrostatic capacity.  He proved the necessary and
sufficient conditions for  existence of a solution,
which are unexpected identical to the corresponding conditions in
the classical Minkowski problem.  Uniqueness was settled by
Caffarelli, Jerison and Lieb  \cite{BMineqCaffarelliJerison}. The
regularity part of the proof depends on the ideas of Caffarelli
\cite{Minkproblem(Caffarelli2)} for regularity of solutions to
Monge-Amp$\rm{\grave{e}}$re equation.

\vskip 5pt
Jerison's work inspired much subsequent research on this topic. In
the very recent article \cite{capacitaryBMtheory(CNSXYZ)}, the
authors (CNSXYZ) extended Jerison's  work
to electrostatic  $\mathfrak{p}$-capacity. Let $K,L$ be  convex bodies in
$\mathbb{R}^n$ and $1<\mathfrak{p}<n$.
CNSXYZ   established the
Hadamard variational formula for $\mathfrak p$-capacity

\begin{equation}\label{HadamardVariationalFormulaPCap}
{\left. {\frac{{d{{\rm{C}}_\mathfrak{p}}(K + tL)}}{{dt}}} \right|_{t
= {0^ + }}} = (\mathfrak{p} - 1)\int\limits_{{S^{n - 1}}}
{{h_L}{{(\xi )}}d{\mu _\mathfrak{p}}(K,\xi )}
\end{equation}
and therefore the Poincar$\rm{\acute{e}}$ $\mathfrak{p}$-capacity formula
\begin{equation}\label{PoincareFormulaPCap}
{\rm C}_{\mathfrak{p}}(K)=\frac{\mathfrak{p}-1}{n-\mathfrak{p}}\int\limits_{\mathbb{S}^{n-1}}h_{K}(\xi)d\mu_{\mathfrak{p}}(K,\xi).
\end{equation}
Here, the new measure $\mu_\mathfrak{p}(K,\cdot)$ is a
finite Borel measure on $\mathbb{S}^{n-1}$,
called the \emph{electrostatic $\mathfrak{p}$-capacitary  measure} of $K$. Formula (\ref{HadamardVariationalFormulaPCap})
suggests that  $\mu_\mathfrak{p}(K,\cdot)$ can be viewed as the differential of $\mathfrak p$-capacity functional.

Consequently, the Minkowski problem for $\mathfrak{p}$-capacity was posed \cite{capacitaryBMtheory(CNSXYZ)}: \emph{Given a finite Borel
measure $\mu$ on $\mathbb{S}^{n-1}$,  what  are the necessary and sufficient
conditions on $\mu$ so that  $\mu$  is the  $\mathfrak{p}$-capacitary  measure $\mu_\mathfrak{p}(K,\cdot)$ of a convex body $K$ in
$\mathbb{R}^n$?}    CNSXYZ proved the uniqueness of the solution  when $1<\mathfrak{p}<n$, and existence and regularity
when $1<\mathfrak{p}<2$.

\subsection{$\boldsymbol L_p$ $\mathfrak{p}$-capacitary measure and $\boldsymbol L_p$ Minkowski problem for $\mathfrak{p}$-capacity}

By reviewing the  Minkowski problems for volume and capacity respectively, we find that they have been intensively
investigated along two parallel tracks, and their similarities  are more highlighted therein.
However, compared with a series of remarkable  results on $L_{p}$ Minkowski problem for volume,
the  general $L_{p}$   Minkowski problem for  capacity is  hardly ever proposed yet.
The time is ripe to initiate the research  on  general $L_{p}$   Minkowski problem for
capacity.

In this paper, we generalize the  Minkowski problem for $\mathfrak{p}$-capacity
to general $L_{p}$  Minkowski problem for $\mathfrak{p}$-capacity. In this sense,
this is the first paper to push the Minkowski problem for $\mathfrak{p}$-capacity to
$L_{p}$ stage. Here, it is worth mentioning that to comply with
the habits, we stick to using the terminology ``$L_p$" Minkowski
problem in our paper. But to avoid the confusion, we use
``$\mathfrak{p}$-capacity", instead of ``$p$-capacity", to
distinguish the ``$p$" in ``$L_p$".

In light of the fundamental significance of  $L_p$ surface area
measures $S_p(K,\cdot)$ in $L_p$  theory for convex bodies,
we introduce the important geometric measure:  \emph{$L_p$ $\mathfrak p$-capacitary measure}.

\vskip 10pt \noindent \textbf{Definition.} Let $p\in\mathbb{R}$ and
$1<\mathfrak{p}<n$.  Suppose  $K$ is a convex body in $\mathbb{R}^n$ with
the origin in its interior. The
\emph{$L_p$ $\mathfrak p$-capacitary measure $\mu_{p,\mathfrak{p}}(K,\cdot)$ of $K$}  is a finite Borel measure
on  $\mathbb{S}^{n-1}$ defined, for Borel $\omega\subseteq {\mathbb S}^{n-1}$, by
\[\mu_{p,\mathfrak{p}}(K,\omega)=\int\limits_{\omega}h_K(\xi)^{1-p}d\mu_{\mathfrak{p}}(K,\xi). \]

Soon Later, it will see that like the $L_p$ surface area
measures $S_p(K,\cdot)$, the $L_p$ $\mathfrak p$-capacitary measure $\mu_{p,\mathfrak{p}}(K,\cdot)$ is
resulted from the variation of $\mathfrak{p}$-capacity functional of $L_p$ sum of convex bodies. Specifically,
if $K,L$ are convex bodies in $\mathbb{R}^n$ with origin in their interiors, then
\[{\left. {\frac{{d{{\rm{C}}_\mathfrak{p}}(K{ + _p}t \cdot_{p} L)}}{{dt}}}
\right|_{t = {0^ + }}} = \frac{(\mathfrak{p} -
1)}{p}\int\limits_{{\mathbb{S}^{n - 1}}} {{h_L}{{(\xi )}^p}d{\mu_{p,\mathfrak{p}}}(K,\xi )},\]
where $1\leq p < \infty$.  See Corollary \ref{LqVariationalFormulaOfCap2}  for details.

Naturally, we pose the  $L_{p}$  Minkowski problem for $\mathfrak{p}$-capacity.

\vskip 10pt \noindent \textbf{$\boldsymbol{L_{p}}$ Minkowski problem
for $\mathfrak{p}$-capacity.} \emph{Suppose $\mu$ is a finite Borel measure on
$\mathbb{S}^{n-1}$,
$1<\mathfrak{p}<n$ and $p\in\mathbb{R}$. What are the necessary and sufficient conditions
on $\mu$ so that $\mu$ is the $L_p$ $\mathfrak p$-capacitary measure $\mu_{p,\mathfrak{p}}(K,\cdot)$
of a convex body $K$ in $\mathbb{R}^n$?}

\vskip 10pt
Jerison \cite{Minkproblem(Jerison)} solved the classical case when $p=1$ and $\mathfrak{p}=2$.
CNSXYZ \cite{capacitaryBMtheory(CNSXYZ)} studied the case when $p=1$ and $1<\mathfrak{p}<n$.
For the general case when  $p\neq 1$,  the corresponding problem  is completely \emph{new}.

\subsection{Main results}

To state our main results, we need to explain something first.
When $p+\mathfrak{p}=n$, the $L_{n-\mathfrak{p}}$ Minkowski problem for
$\mathfrak{p}$-capacity is a bit troubling, since  two
convex bodies with the same $L_{n-\mathfrak{p}}$
$\mathfrak{p}$-capacitary measure are dilates each other, but not necessarily
identical.  For simplicity, we technically
normalize the $L_p$ Minkowski problem for $\mathfrak{p}$-capacity as
folows:  \emph{Under what necessary and sufficient conditions on $\mu$
does there exist a  convex body $K^*$ so that ${\rm
C}_{\mathfrak{p}}(K^*)^{-1}\mu_{p, \mathfrak{p}}(K^*,\cdot)=\mu$?}
Note that when $p+\mathfrak{p}\neq n $, two problems are essentially
equivalent, in the sense that $K={\rm C}_{\mathfrak{p}}(K^*)^{1/(p+{\mathfrak{p}}-n)}K^*$.

In this article,  we  solve the discrete $L_{p}$
Minkowski problem for $\mathfrak{p}$-capacity when $1<\mathfrak{p}<n$, and the
general $L_{p}$ Minkowski problem for $\mathfrak{p}$-capacity when $1<\mathfrak{p}\le 2$.

\begin{theorem}\label{Theorem1.1}
Suppose $1<p<\infty$ and $1<\mathfrak{p}<n$. If $\mu$ is a discrete
measure on $\mathbb{S}^{n-1}$ which is not concentrated on any
closed hemisphere, then there exists a unique polytope $P$ with the
origin in its interior, such that
\[\mu_{p,{\mathfrak p}}(P,\cdot)=c\mu,\]
where $c=1$ if $p+\mathfrak{p} \neq n $, or
${\rm C}_{\mathfrak p}(P)$ if $p+\mathfrak{p}=n$.
Furthermore,  $P$ is origin-symmetric if $\mu$ is even.
\end{theorem}

\begin{theorem}\label{Theorem1.2}
Suppose $1<p<\infty$ and $1<\mathfrak{p}\le 2$. If $\mu$ is a finite
Borel measure on $\mathbb{S}^{n-1}$ which is not concentrated on any
closed hemisphere, then there exists a unique convex body $K$
containing the origin, such that
\[d\mu_{\mathfrak{p}}(K,\cdot)= c h_K^{p-1}d\mu,\]
where $c=1$  if $p+\mathfrak{p} \neq n$, or ${\rm C}_{\mathfrak p}(K)$ if $p+\mathfrak{p}=n$.
Furthermore, $K$  contains the origin in its interior if  $p\ge n$. Therefore,
$\mu_{p,\mathfrak{p}}(K,\cdot)=\mu$.
\end{theorem}

\begin{theorem}\label{Theorem1.3}
Suppose $1<p<\infty$ and $1<\mathfrak{p}\le 2$. If $\mu$  is a
finite even Borel measure on $\mathbb{S}^{n-1}$ which is not
concentrated on any great subsphere,  there exists a unique
origin-symmetric convex body $K$ such that $\mu_{p,
\mathfrak{p}}(K,\cdot)=c\mu$, where $c=1$ if
$p+\mathfrak{p} \neq n$, or ${\rm C}_{\mathfrak p}(K)$ if
$p+\mathfrak{p}=n$.
\end{theorem}

Continuity of the solution to the $L_p$ Minkowski problem for $\mathfrak{p}$-capacity is shown.

\begin{theorem}\label{Theorem1.4}
Suppose that $1<p<\infty$ and $1<\mathfrak{p}\le 2$.  Let $\mu$ and
$\mu_j$, $j\in\mathbb{N}$,  be finite Borel measures on
$\mathbb{S}^{n-1}$ which are not concentrated on any closed
hemisphere, and  $K$ and $K_j$ be convex bodies containing the
origin such that ${\rm C}_{\mathfrak{p}}(K)^{-1}\mu_{p,
\mathfrak{p}}(K,\cdot)=\mu$ and ${\rm
C}_{\mathfrak{p}}(K_j)^{-1}\mu_{p,\mathfrak{p}}(K_j,\cdot)= \mu_j$,
respectively. If $\mu_j\to\mu$ weakly,  then $ K_j\to K$, as
$j\to\infty$.
\end{theorem}

 CNSXYZ \cite{capacitaryBMtheory(CNSXYZ)} demonstrated
the weak convergence of $L_p$ $\mathfrak{p}$-capacitary measure: If
$K_j\to K$,  then $\mu_{p,\mathfrak{p}}(K_j,\cdot)\to
\mu_{p,\mathfrak{p}}(K,\cdot)$ weakly. Theorem \ref{Theorem1.4}
shows that the converse still holds: If
$\mu_{p,\mathfrak{p}}(K_j,\cdot)\to \mu_{p,\mathfrak{p}}(K,\cdot)$
weakly, then $K_j\to K$.

\vskip 10pt We emphasize that, for  $p>1,$ the $L_{p} $ Minkowski
problem for $\mathfrak{p}$-capacity is considerably more complicated
than the $p=1$ case, requiring both new ideas and techniques. Our
approach to this problem is rooted in the ideas and techniques from
convex geometry. So its proof exhibits  rich geometric flavour.
Specifically, to  prove Theorem 1.1, techniques developed by  Hug
and LYZ \cite{Minkproblem(HugLYZ)},  Klain \cite{Minkproblem(Klain)}
and  Lutwak \cite{LpBM(Lutwak1),LpBM(LYZ5)} are comprehensively
employed. In addition, techniques developed by the authors
themselves in \cite{LpBM(ZouXiong1), LpBM(ZouXiong2),
LpBM(ZouXiong3), LpBM(ZouXiong4)} are also crucial to  the proof. To
prove  Theorem 1.2,  we turn the Minkowski problem  into  solving
two dual optimization problems. This strategy was fist used by LYZ
\cite{LpBM(LYZ7)} to establish the $L_p$ John ellipsoids, and then
developed by Zou and Xiong  to establish the Orlicz-John ellipsoids
\cite{LpBM(ZouXiong1)} and the Orlicz-Legendre ellipsoids \cite{
LpBM(ZouXiong3)}.

\vskip 10pt

This paper is organized as follows. In Section 2, we introduce
necessary notations and collect some basic facts concerning the
convex bodies, the $\mathfrak{p}$-capacity and the Aleksandrov
bodies. Some basic facts of the $L_p$ $\mathfrak{p}$-capacitary
measures $\mu_{p,\mathfrak{p}}(K,\cdot)$ are provided in Section 3.
For example, to study the uniqueness of the $L_p$ Minkowski problem
for $\mathfrak{p}$-capacity, we prove the $L_p$ Minkowski inequality
for $\mathfrak{p}$-capacity and then characterize the uniqueness of
$\mu_{p,\mathfrak{p}}(K,\cdot)$. The proof of Theorem
\ref{Theorem1.1} is provided in Section 4. Along with the arguments
in Section 4, we show that Theorem 1.1 still holds when
$p=1$ in Section 5, which solves CNSXYZ's \cite[p.
1517]{capacitaryBMtheory(CNSXYZ)} open problem for discrete
measures. Theorem \ref{Theorem1.2} and Theorem \ref{Theorem1.4} are
provided in Section 8 and Section 9, respectively. To prove these
theorems, we reformulate the $L_p$ Minkowski problem for $\mathfrak
p$-capacity into a pair of dual optimization problems. See Section 6
for details. More preliminaries about these optimization problems
are provided in Section 7.

\vskip 25pt
\section{\bf Preliminaries}
\vskip 8pt

\subsection{Basics of convex bodies}
For quick reference, we collect some basic facts on convex bodies.
Excellent references are the books by Gardner
\cite{BookGardner}, Gruber \cite{BookGrubder} and Schneider
\cite{BookSchneider}.

As usual, write $x\cdot y$  for the standard inner product of
$x,y\in\mathbb{R}^n$. Each compact convex set $K$ in $\mathbb{R}^n$
is uniquely determined by its \emph {support function} $h_K:
\mathbb{R}^n\to\mathbb{R}$, which is defined by $
h_K(x)=\max\left\{x\cdot y: y\in K \right\}$, for  $x\in\mathbb{R}^n$.
It is easily seen that the support function is
positively homogeneous of order $1$.

The class of  compact convex sets in $\mathbb{R}^n$ is often
equipped with the \emph{Hausdorff metric} $\delta_H$, which is defined for compact convex sets $K$ and $L$
by
\[\delta_H(K,L)=\max\left\{\mid h_K(\xi)-h_L(\xi)\mid: \xi\in \mathbb{S}^{n-1}\right\}.\]

Write $\mathcal{K}^n$ for the set of convex bodies in
$\mathbb{R}^n$, and write $\mathcal{K}^n_o$ for the set of convex
bodies with the origin $o$ in their interiors. Let $K$ and $L$ be
compact convex sets. For $s>0$,  the set $s K=\{s x: x\in K\}$ is
called a \emph{dilate} of $K$.  $K$ and $L$ are said to be
\emph{homothetic}, provided $K=sL+x$, for some $s>0$ and
$x\in\mathbb{R}^n$. The \emph{reflection} of $K$  is the set
$-K=\{-x:x\in K\}$. The \emph{Minkowski sum} of $K$ and $L$ is the
set $K+L=\{x+y: x\in K, y\in L\}$.

Given $1\le p < \infty$, $K\in\mathcal{K}^n_o$ and $t>0$, the $L_p$
scalar multiplication $t\cdot  K$ is the convex body
$t^{\frac{1}{p}} K$. For  $K,L\in\mathcal{K}^n_o$, their \emph{$L_p$
sum} (See, e.g., \cite{LpBM(Firey), LpBM(Lutwak1),
LpBM(GardnerHugWeil), LpBM(LYZ11)}) is the convex body $K +_p L$
defined by
\[ h_{K +_p L}(\xi)^p =h_{K}(\xi)^p +h_{L}(\xi)^p, \quad  \xi\in \mathbb{S}^{n-1}.\]
Clearly, $K+_1 L=K+L$.

Let $C(\mathbb{S}^{n-1})$ be the set of  continuous real functions
on $\mathbb{S}^{n-1}$, equipped with the metric induced by the
maximal norm. Let $C_+(\mathbb{S}^{n-1})$ be the subset of
$C(\mathbb{S}^{n-1})$, consisting of strictly positive functions.

For $f,g\in C_+(\mathbb{S}^{n-1})$ and $t>0$, define
\[h+_p t\cdot f=\left(h^p+tf^p\right)^{\frac{1}{p}}.\]
For brevity, write $h+_p f$ for $h+_p 1\cdot f$.

For $f\in C_+(\mathbb{S}^{n-1})$, define
\[ {\scriptstyle \pmb
\lbrack} f {\scriptstyle \pmb \rbrack} =\bigcap\limits_{\xi\in
\mathbb{S}^{n-1} }\{x\in\mathbb{R}^n: x\cdot \xi\le f(\xi)\}. \] The
set ${\scriptstyle \pmb \lbrack} f {\scriptstyle \pmb \rbrack}$ is
called the \emph{Aleksandrov body} (also known as \emph{Wulff
shape}) associated with $f$.  Obviously, ${\scriptstyle \pmb
\lbrack} f {\scriptstyle \pmb \rbrack}$ is a convex body with the
origin in its interior.

\subsection{ Basics of $\mathfrak{p}$-capacity}

In this part, some basics of $\mathfrak{p}$-capacity are listed. For
more details on $\mathfrak{p}$-capacity, see, e.g.,
 \cite{BookEvansGariepy,capacitaryBMtheory(CNSXYZ),BookGilbargTrudinger,Minkproblem(Jerison)}.

Let $1<\mathfrak{p}<n$.  The $\mathfrak{p}$-capacity ${\rm
C}_{\mathfrak{p}}$ is  increasing with respect to the inclusion of
sets. That is, if $E\subseteq F$, then ${\rm C}_{\mathfrak{p}}(E)\le
{\rm C}_{\mathfrak{p}}(F)$. The $\mathfrak{p}$-capacity ${\rm
C}_{\mathfrak{p}}$ is positively homogeneous of order
$(n-\mathfrak{p})$, i.e., ${\rm
C}_{\mathfrak{p}}(sE)=s^{n-\mathfrak{p}}{\rm C}_{\mathfrak{p}}(E)$,
for $s>0$. Also, it is rigid invariant, i.e., ${\rm
C}_{\mathfrak{p}}(OE+x)={\rm C}_{\mathfrak{p}}(E)$, for
$x\in\mathbb{R}^n$ and $O\in \rm{O}(n)$.

For $K\in\mathcal{K}^n$,   the $\mathfrak{p}$-capacitary measure
$\mu_{\mathfrak{p}}(K,\cdot)$ is   positively homogeneous of order
$(n-\mathfrak{p}-1)$, i.e.,
$\mu_{\mathfrak{p}}(sK,\cdot)=s^{n-\mathfrak{p}-1}\mu_{\mathfrak{p}}(K,\cdot)$,
for  $s>0$. For $x\in\mathbb{R}^n$,
$\mu_{\mathfrak{p}}(K+x,\cdot)=\mu_{\mathfrak{p}}(K,\cdot)$,  i.e., it is translation invariant. The centroid of $\mu_{\mathfrak{p}}(K,\cdot)$ is
at the origin, i.e., $\int_{\mathbb{S}^{n-1}}\xi
d\mu_{\mathfrak{p}}(K,\xi)=o.$

The weak convergence  of $\mathfrak{p}$-capacitary measures was
proved by CNSXYZ \cite[p. 1550]{capacitaryBMtheory(CNSXYZ)}: If
$\{K_j\}_{j\in\mathbb{N}}\subset \mathcal{K}^n$ converges to
$K\in\mathcal{K}^n$, then $\{\mu_{\mathfrak{p}}(K_j,\cdot)\}_j$
converges weakly to $\mu_{\mathfrak{p}}(K,\cdot)$.

Let $K\in\mathcal{K}^n_o$ and $f\in C(\mathbb{S}^{n-1})$. There
exists $t_0>0$ such that $ h_K+t f\in C_+(\mathbb{S}^{n-1})$, for
$|t|<t_0$. So, there is a continuous family of Aleksandrov bodies $
{\scriptstyle \pmb \lbrack} h_K+t f {\scriptstyle \pmb \rbrack}$
with $|t|<t_0$. The Hadamard variational formula for
$\mathfrak{p}$-capacity (see \cite[p.
1547]{capacitaryBMtheory(CNSXYZ)}) states that
\begin{equation}\label{HadamardVariationalFormula}
{\left. {\frac{d{{\rm C}_{\mathfrak{p}}}({\scriptstyle \pmb \lbrack}
h_K+t f {\scriptstyle \pmb \rbrack})}{{dt}}} \right|_{t = 0}} =
(\mathfrak{p} - 1)\int\limits_{{\mathbb{S}^{n - 1}}} {f(\xi
)d{\mu_{\mathfrak{p}}}(K,\xi )} .
\end{equation}

For  $K, L\in\mathcal{K}^n$, the \emph{mixed
$\mathfrak{p}$-capacity} $C_{\mathfrak{p}}(K,L)$ (see \cite[p.
1549]{capacitaryBMtheory(CNSXYZ)})   is defined by
\begin{equation}\label{mixedcapacity}
{{\rm C}_{\mathfrak{p}}}(K,L) = \frac{1}{{n - \mathfrak{p}}}{\left.
{\frac{d{{\rm C}_{\mathfrak{p}}}(K + tL)}{{dt}}} \right|_{t = {0^ +
}}} = \frac{{\mathfrak{p} - 1}}{{n -
\mathfrak{p}}}\int\limits_{{\mathbb{S}^{n - 1}}} {{h_L}(\xi
)d{\mu_{\mathfrak{p}}}(K,\xi )} .
\end{equation}
When $L=K$, it  reduces to the Poincar$\rm{\acute{e}}$
$\mathfrak{p}$-capacity formula (1.4). From  the weak convergence of
$\mathfrak{p}$-capacitary measures, it follows that ${\rm
C}_{\mathfrak{p}}(K,L)$ is continuous in $(K,L)$.

The $\mathfrak{p}$-capacitary Brunn-Minkowski inequality, proved by  Colesanti and Salani
\cite{capacitaryBMineq(ColesantiSalani)}, reads: If
$K,L\in\mathcal{K}^n$, then
\begin{equation}
{\rm C}_{\mathfrak{p}}(K+L)^{\frac{1}{n-\mathfrak{p}}}\ge {\rm
C}_{\mathfrak{p}}(K)^{\frac{1}{n-\mathfrak{p}}}
+C_{\mathfrak{p}}(L)^{\frac{1}{n-\mathfrak{p}}},
\end{equation}
with equality if and only if $K$ and $L$ are homothetic. When
$\mathfrak{p}=2$, the inequality was first established by Borell
\cite{capacitaryBMineq(Borell)}, and the equality condition was
shown by Caffarelli, Jerison and Lieb
\cite{BMineqCaffarelliJerison}. For more deatils, see, e.g.,
Colesanti \cite{capacitaryBMineq(Colesanti)}, Gardner
\cite{BMineq(Gardner)}, and Gardner and Hartenstine
\cite{capacitaryBMineq(Gardner)}.

The $\mathfrak{p}$-capacitary
Brunn-Minkowski inequality is equivalent to the
$\mathfrak{p}$-capacitary Minkowski inequality,
\begin{equation}
{\rm C}_{\mathfrak{p}}(K,L)\ge {\rm
C}_{\mathfrak{p}}(K)^{n-\mathfrak{p}-1}{\rm
C}_{\mathfrak{p}}(L),\end{equation} with equality if and only if $K$
and $L$ are homothetic. See \cite[p.
1549]{capacitaryBMtheory(CNSXYZ)} for its proof.

\subsection{Basics of Aleksandrov bodies}

For  $f\in C_+(\mathbb{S}^{n-1})$, define
\begin{equation}\label{DefOfCapOfFunctions}
{\rm C}_{\mathfrak{p}}(f)={\rm C}_{\mathfrak{p}}({\scriptstyle \pmb
\lbrack} f {\scriptstyle \pmb \rbrack}).
\end{equation}
Obviously, ${\rm C}_{\mathfrak{p}}(h_K)={\rm C}_{\mathfrak{p}}(K)$,
for  $K\in\mathcal{K}^n_o$.

The \emph{Aleksandrov convergence lemma} reads: If the sequence
$\{f_j\}_j \subset C_+(\mathbb{S}^{n-1})$ converges uniformly to
$f\in  C_+(\mathbb{S}^{n-1})$, then $\lim_{j\to\infty} {\scriptstyle
\pmb \lbrack} f_j {\scriptstyle \pmb \rbrack}={\scriptstyle \pmb
\lbrack} f {\scriptstyle \pmb \rbrack}$. From this lemma and the
continuity of ${\rm C}_{\mathfrak{p}}$ on $\mathcal{K}^n$, we see
that
 ${\rm C}_{\mathfrak{p}}: C_+(\mathbb{S}^{n-1})\to (0,\infty)$ is continuous.

Let  $1\le p<\infty$ and $1<\mathfrak{p}<n$. For $K\in\mathcal{K}^n_o$ and
nonnegative $f\in C(\mathbb{S}^{n-1})$, define
\begin{equation}\label{DefOfMixedCapOfFunctions}
{\rm C}_{p,\mathfrak{p}}(K,f) = \frac{{\mathfrak{p} - 1}}{{n -
\mathfrak{p}}}\int\limits_{{\mathbb{S}^{n - 1}}} {f{{(\xi
)}^p}{h_K}{{(\xi )}^{1 - p}}d{\mu_{\mathfrak{p}}}(K,\xi )}.
\end{equation}
For brevity, write ${\rm C}_{\mathfrak{p}}(K,f)$ for ${\rm
C}_{1,\mathfrak{p}}(K,f)$. Obviously, ${\rm
C}_{p,\mathfrak{p}}(K,h_K)={\rm C}_{\mathfrak{p}}(K)$.

\begin{lemma}\label{Alekidentity}
 Suppose  $1\le p<\infty$ and $1<\mathfrak{p}<n$. If $f\in C_+(\mathbb{S}^{n-1})$, then
\[{\rm C}_{p,\mathfrak{p}}({\scriptstyle \pmb \lbrack} f {\scriptstyle \pmb \rbrack},f)
={\rm C}_{\mathfrak{p}}({\scriptstyle \pmb \lbrack} f {\scriptstyle
\pmb \rbrack})={\rm C}_{\mathfrak{p}}(f). \]
\end{lemma}

 \begin{proof}
Note that $h_{{\scriptstyle \pmb \lbrack} f {\scriptstyle \pmb
\rbrack}}\le f$. A basic fact established by Aleksandrov is that
$h_{{\scriptstyle \pmb \lbrack} f {\scriptstyle \pmb \rbrack}}=f$,
a.e.  with respect to $S_{{\scriptstyle \pmb \lbrack} f
{\scriptstyle \pmb \rbrack}}$. That is, $S_{{\scriptstyle \pmb
\lbrack} f {\scriptstyle \pmb \rbrack}}(\{h_{{\scriptstyle \pmb
\lbrack} f {\scriptstyle \pmb \rbrack}}< f\})=0$. Since
$\mu_{\mathfrak{p}}({\scriptstyle \pmb \lbrack} f {\scriptstyle \pmb
\rbrack},\cdot)$ is absolutely continuous with respect to
$S_{{\scriptstyle \pmb \lbrack} f {\scriptstyle \pmb \rbrack}}$, it
follows that $ \mu_{\mathfrak{p}}({\scriptstyle \pmb \lbrack} f
{\scriptstyle \pmb \rbrack},\{h_{{\scriptstyle \pmb \lbrack} f
{\scriptstyle \pmb \rbrack}}< f\}) =0$. Combining this fact and the
inequality $h_{{\scriptstyle \pmb \lbrack} f {\scriptstyle \pmb
\rbrack}}\le f$, it follows that
\[{\rm C}_{p,\mathfrak{p}}({\scriptstyle \pmb \lbrack} f {\scriptstyle \pmb
\rbrack}, f) - {{\rm C}_{\mathfrak{p}}}(f)= \frac{{\mathfrak{p} -
1}}{{n - \mathfrak{p}}}\int\limits_{\{ f > {h_{{\scriptstyle \pmb
\lbrack} f {\scriptstyle \pmb \rbrack}}}\} } {\left( {f^p -
{h^p_{{\scriptstyle \pmb \lbrack} f {\scriptstyle \pmb \rbrack}}}}
\right){h^{1 - p}_{{\scriptstyle \pmb \lbrack} f {\scriptstyle \pmb
\rbrack}}}d{\mu_{\mathfrak{p}}}({\scriptstyle \pmb \lbrack} f
{\scriptstyle \pmb \rbrack},\xi )}= 0, \] as desired.
\end{proof}

Note that for $K\in\mathcal{K}^n_o$ and $f\in
C_+(\mathbb{S}^{n-1})$, we have ${\rm
C}_{\mathfrak{p}}(K,h_{{\scriptstyle \pmb \lbrack} f {\scriptstyle
\pmb \rbrack}})\le {\rm C}_{\mathfrak{p}}(K,f)$.

\vskip 25pt
\section{\bf The $\boldsymbol{L_p}$ $\boldsymbol{\mathfrak{p}}$-capacitary measure $\boldsymbol{\mu_{p,\mathfrak{p}}(K,\cdot)}$}
\vskip 10pt

\subsection{The first $\boldsymbol{L_p}$ variational of  $\mathfrak{p}$-capacity}

\begin{lemma}\label{variationalFormulaCNSXYZ}
Let $I\subset\mathbb{R}$ be an interval containing both $0$ and some
positive number, and let $h_t(\xi)=h(t,\xi): I\times {\mathbb
S}^{n-1}\to (0,\infty)$ be continuous, such that the convergence in
\[
h'(0,\xi)=\lim_{t\to 0}\frac{h(t,\xi)-h(0,\xi)}{t}
\]
is uniform on ${\mathbb S}^{n-1}$. Then
\[\lim_{t\to 0^+}\frac{ {\rm C}_{\mathfrak p}(h_t)-{\rm C}_{\mathfrak p}(h_0) }{t}
= ({\mathfrak{p}} - 1)\int\limits_{{\mathbb{S}^{n - 1}}} {h'(0,\xi )d{\mu _{\mathfrak{p}}}({\scriptstyle \pmb \lbrack} h_0 {\scriptstyle
\pmb \rbrack},\xi )} .\]
\end{lemma}

\begin{lemma}\label{LqVariationalFormulaOfCap1}
Suppose $1\leq p<\infty$ and $1<\mathfrak{p}<n$. If $K\in
\mathcal{K}^n_o$ and $f\in C(\mathbb{S}^{n-1})$ is nonnegative, then
\[{\left. {\frac{{d{{\rm{C}}_\mathfrak{p}}({h_K}{ + _p}t \cdot f)}}{dt}} \right|_{t = 0^+}} = \frac{{n - \mathfrak{p}}}{p}{\rm{C}}_{p,\mathfrak{p}}(K,f).\]
\end{lemma}

\begin{proof}
Take an interval $I=[0,t_0]$ for $0<t_0<\infty$. Since
$h_t(\xi)=h(t,\xi)=(h_K+_p t\cdot f)(\xi): I\times {\mathbb
S}^{n-1}\to (0,\infty)$ is continuous, and
\[\lim_{t\to 0^+}\frac{{({h_K}{ +_p}t \cdot f) - {h_K}}}{t} =
\frac{{{f^p}h_K^{1 - p}}}{p}\]
 uniformly on $\mathbb{S}^{n-1}$, the desired lemma is a consequence of Lemma \ref{variationalFormulaCNSXYZ} and (\ref{DefOfMixedCapOfFunctions}).
\end{proof}

Note that when $p=1$,  Lemma \ref{LqVariationalFormulaOfCap1}
reduces to the Hadamard variational formula
(\ref{HadamardVariationalFormula}).

\begin{corollary}\label{LqVariationalFormulaOfCap2}
Suppose $1\leq p<\infty$ and $1<\mathfrak{p}<n$.
If  $K\in\mathcal{K}^n_o$ and $L$ is a compact convex set containing
the origin,  then
\[{\left. {\frac{d {{\rm C}_{\mathfrak{p}}}(K{ +_p}t \cdot L)}{{dt }}} \right|_{t = {0^ + }}} = \frac{{\mathfrak{p} - 1}}{p}\int\limits_{{\mathbb{S}^{n - 1}}}{h_L}{{(\xi )}^p} {{h_K}{{(\xi )}^{1 - p}}d{\mu_{\mathfrak{p}}}(K,\xi )} .\]
\end{corollary}

Let $1<\mathfrak{p}<n$. Now, we can introduce the following
definitions.
\begin{definition}\label{DefLqMixedCap}
If $1\le p<\infty$,  $K\in\mathcal{K}^n_o$ and $L$ is a compact
convex set containing the origin, then the quantity ${{\rm
C}_{p,\mathfrak{p}}}(K,L)$ defined by
\[{{\rm C}_{p,\mathfrak{p}}}(K,L) = \frac{{\mathfrak{p} - 1}}{{n - \mathfrak{p}}}\int\limits_{{\mathbb{S}^{n - 1}}}{h_L}{{(\xi )}^p} {{h_K}{{(\xi )}^{1 - p}}d{\mu_{\mathfrak{p}}}(K,\xi )},\]
is called the $L_p$ \emph{mixed}  $\mathfrak{p}$-\emph{capacity} of
$K$ and $L$.
\end{definition}

\begin{definition}\label{DefOfLqCapMeasure}
If $p\in\mathbb{R}$ and $K\in\mathcal{K}^n_o$, then the Borel
measure $\mu_{p, \mathfrak{p}}(K,\cdot)$ on $\mathbb{S}^{n-1}$,
defined by
\[\mu_{p,\mathfrak{p}}(K,\omega)=\int\limits_{\omega}h_K^{1-p}d\mu_{\mathfrak{p}}(K,\cdot),\]
for Borel $\omega\subseteq \mathbb{S}^{n-1}$, is called the $L_p$
$\mathfrak{p}$-\emph{capacitary measure} of $K$.
\end{definition}
Obviously,
${\rm C}_{1,\mathfrak{p}} (K,L)={\rm
C}_{\mathfrak{p}}(K,L)$,
${\rm C}_{p,\mathfrak{p}}(K,K)={\rm
C}_{\mathfrak{p}}(K)$ and ${\rm C}_{p,\mathfrak{p}}(K,h_L)={\rm
C}_{p,\mathfrak{p}}(K,L)$. Also,
$\mu_{1,\mathfrak{p}}(K,\cdot)=\mu_{\mathfrak{p}}(K,\cdot)$,
$\frac{{\mathfrak p}-1}{n-{\mathfrak p}}
\mu_{0,\mathfrak{p}}(K,\mathbb{S}^{n-1})={\rm{C}}_\mathfrak{p}(K)$.
In
addition, ${\rm C}_{p,\mathfrak{p}}(OK,OL)={\rm
C}_{p,\mathfrak{p}}(K,L)$, for $O\in {\rm{O}}(n)$.

As the $L_p$ mixed volume $V_{p}(K,L)$ and the $L_p$ surface area
measure $S_{p}(K,\cdot)$ greatly extend the first mixed volume
$V_{1}(K,L)$ and the classical surface area measure $S(K,\cdot)$ in
convex geometry, respectively, ${\rm C}_{p,\mathfrak{p}}(K,L)$ and
$\mu_{p,\mathfrak{p}}(K,\cdot)$ are precisely the $L_p$ extensions
of the mixed $\mathfrak{p}$-capacity ${\rm C}_{\mathfrak{p}}(K,L)$
and  the $\mathfrak{p}$-capacitary measure
$\mu_{\mathfrak{p}}(K,\cdot)$, respectively.

The next lemma shows that ${\rm C}_{p,\mathfrak{p}}(K,L)$ is
continuous in $(K,L,p)$.

\begin{lemma}
Suppose that $K_i, L_i, K, L \in \mathcal{K}^n_o$,  $ p_i,p\in[1,\infty)$, $i\in\mathbb{N}$ and $1<\mathfrak{p}<n$.
 If $(K_i,L_i)\to (K,L)$ and $p_i\to p$, as $i\to\infty$, then $ {\rm
C}_{p_i,\mathfrak{p}}(K_{i},L_{i}) \to {\rm
C}_{p,\mathfrak{p}}(K,L)$.
\end{lemma}

\begin{proof}
Since  $h_{K_i}, h_{L_i}>0$ and  $h_{K_i}\to h_{K}$, $h_{L_i}\to
h_L$ uniformly on $\mathbb{S}^{n-1}$,  it follows that
${h_{L_i}}/{h_{K_i}}\to {h_L}/{h_K}$ uniformly on
$\mathbb{S}^{n-1}$. Clearly, there exists a compact interval
$I\subset (0,\infty)$, such that ${h_{L_i}}/{h_{K_i}}\in I$ for  all
$i$. Since the sequence
 $t^{p_i}$ converges uniformly to
$t^p$ on $I$, it follows that   $
\left({h_{L_i}}/{h_{K_i}}\right)^{p_i}\to \left(
{h_L}/{h_K}\right)^p$, uniformly on $\mathbb{S}^{n-1}$. Meanwhile,
the convergence $K_i\to K$ implies that
$\mu_{\mathfrak{p}}(K_i,\cdot)\to\mu_{\mathfrak{p}}(K,\cdot)$
weakly. By Definition \ref{DefLqMixedCap}, the desired limit is
obtained.
\end{proof}

The weak convergence of $\mathfrak{p}$-capacitary measures implies
the weak convergence of $\mu_{p,\mathfrak{p}}$.
\begin{lemma}
Suppose that $K_i,K\in \mathcal{K}^n_o$, $i\in\mathbb{N}$, $1\leq
p<\infty$ and  $1<\mathfrak{p}<n$.  If $K_i\to K$, as $i\to\infty$,
then
$\mu_{p,\mathfrak{p}}(K_i,\cdot)\to\mu_{p,\mathfrak{p}}(K,\cdot)$
weakly.
\end{lemma}

From the $(n-\mathfrak{p}-1)$-order positive homogeneity of
$\mathfrak{p}$-capacitary measures, the  positive homogeneity of support
functions and Definition \ref{DefOfLqCapMeasure}, we obtain the following result.
\begin{lemma}\label{positivehomogeneitymeasure2}
Suppose that $K\in\mathcal{K}^n_o$, $1\le p<\infty$ and $1<\mathfrak{p}<n$.
Then for $s>0$,
\[\mu_{p,\mathfrak{p}}(sK,\cdot)=s^{n-\mathfrak{p}-p}\mu_{p,\mathfrak{p}}(K,\cdot).\]
\end{lemma}

\subsection{$\boldsymbol{L_p}$  Minkowski inequality for $\mathfrak{p}$-capacity}

In this part, we will show that associated with ${\rm
C}_{p,\mathfrak{p}}(K,L)$, there is a natural $L_{p}$ extension of
the  $\mathfrak{p}$-capacitary Minkowski inequality. Then we will
use it to extend the $\mathfrak{p}$-capacitary Brunn-Minkowski
inequality to the $L_p$ stage. It is interesting that the $L_p$
Brunn-Minkowski type inequality for $\mathfrak{p}$-capacity was
previously established in  \cite{LpBM(ZouXiong5)} by the authors'
$L_{p}$ transference principle.

\begin{theorem}\label{p-capMinkowskiineqthm}
Suppose  $1\leq p<\infty$ and $1<\mathfrak{p}<n$. If $K\in\mathcal{K}^n_o$ and $f\in C_+(\mathbb{S}^{n-1})$,  then
\begin{equation}\label{ineq1}
{\rm C}_{p,\mathfrak{p}}{(K,f)^{n - \mathfrak{p}}} \ge {{\rm
C}_{\mathfrak{p}}}{(K)^{n - \mathfrak{p} - p}}{{\rm
C}_{\mathfrak{p}}}{(f)^p},
\end{equation}
with equality if and only if $K$ and ${\scriptstyle \pmb \lbrack} f
{\scriptstyle \pmb \rbrack}$ are dilates.
\end{theorem}

\begin{proof}
From  (\ref{DefOfMixedCapOfFunctions}), (\ref{DefOfCapOfFunctions}) and the H$\rm{\ddot{o}}$lder inequality, it follows that
\begin{align*}
{{\rm C}_\mathfrak{p}}(K,f)
&= \frac{{\mathfrak{p} - 1}}{{n - \mathfrak{p}}}\int\limits_{{\mathbb{S}^{n - 1}}} {f(\xi ){h_K}{{(\xi )}^{ - \frac{{p - 1}}{p}}}{h_K}{{(\xi )}^{\frac{{p - 1}}{p}}}d{\mu _\mathfrak{p}}(K,\xi )}  \\
&\le {\left( {\frac{{\mathfrak{p} - 1}}{{n - \mathfrak{p}}}\int\limits_{{\mathbb{S}^{n - 1}}} {f{{(\xi )}^p}{h_K}{{(\xi )}^{1 - p}}d{\mu _\mathfrak{p}}(K,\xi )} } \right)^{\frac{1}{p}}}{\left( {\frac{{\mathfrak{p} - 1}}{{n - \mathfrak{p}}}\int\limits_{{\mathbb{S}^{n - 1}}} {{h_K}(\xi )d{\mu _\mathfrak{p}}(K,\xi )} } \right)^{\frac{{p - 1}}{p}}} \\
&= {{\rm C}_{p,\mathfrak{p}}}{(K,f)^{\frac{1}{p}}}{{\rm
C}_{\mathfrak{p}}}{(K)^{\frac{{p - 1}}{p}}}.
\end{align*}
Thus,
\[{{\rm C}_{p,\mathfrak{p}}}(K,f) \ge {{\rm C}_\mathfrak{p}}{(K,f)^p}{{\rm C}_\mathfrak{p}}{(K)^{1 - p}}.\]
From this inequality, the fact that ${\rm C}_\mathfrak{p}(K,f)\ge
{\rm C}_\mathfrak{p}(K,{\scriptstyle \pmb \lbrack} f {\scriptstyle
\pmb \rbrack})$ and the $\mathfrak{p}$-capacitary Minkowski
inequality, it follows that
\begin{align*}
{{\rm C}_{p,\mathfrak{p}}}(K,f) &\ge {{\rm
C}_\mathfrak{p}}{(K,{\scriptstyle \pmb \lbrack} f
{\scriptstyle \pmb \rbrack})^p}{{\rm C}_\mathfrak{p}}{(K)^{1 - p}} \\
&\ge {\left( {{{\rm C}_\mathfrak{p}}{{(K)}^{\frac{{n - \mathfrak{p}
- 1}}{{n - \mathfrak{p}}}}}{{\rm C}_\mathfrak{p}}{{({\scriptstyle
\pmb \lbrack} f
{\scriptstyle \pmb \rbrack})}^{\frac{1}{{n - \mathfrak{p}}}}}} \right)^p}{{\rm C}_\mathfrak{p}}{(K)^{1 - p}} \\
&= {{\rm C}_\mathfrak{p}}{(K)^{\frac{{n - \mathfrak{p} - p}}{{n -
\mathfrak{p}}}}}{{\rm C}_\mathfrak{p}}{({\scriptstyle \pmb \lbrack}
f {\scriptstyle \pmb \rbrack})^{\frac{p}{{n - \mathfrak{p}}}}}.
\end{align*}

In the next, we prove the equality condition.

Assume that equality holds in (\ref{ineq1}). By the equality
condition of $\mathfrak{p}$-capacitary Minkowski inequality, there
exist $x\in\mathbb{R}^n$ and $s>0$, such that ${\scriptstyle \pmb
\lbrack} f {\scriptstyle \pmb \rbrack}=sK+x$. Meanwhile, by the
equality condition of the H$\rm{\ddot{o}}$lder inequality, ${{\rm
C}_\mathfrak{p}}(K,{\scriptstyle \pmb \lbrack} f {\scriptstyle \pmb
\rbrack}){h_K(\xi)} = {{\rm C}_\mathfrak{p}}(K){h_{{\scriptstyle
\pmb \lbrack} f {\scriptstyle \pmb \rbrack}}(\xi)}$, for
$\mu_\mathfrak{p}(K,\cdot)$-almost all $\xi\in \mathbb{S}^{n-1}$.
Hence, for $\mu_\mathfrak{p}(K,\cdot)$-almost all $\xi\in
\mathbb{S}^{n-1}$,
\[\left( {s{{\rm C}_\mathfrak{p}}(K)  + \frac{{\mathfrak{p} - 1}}{{n - \mathfrak{p}}}x \cdot \int\limits_{{\mathbb{S}^{n - 1}}} {\xi d{\mu _p}(K,\xi )} } \right){h_K}(\xi ) = {{\rm C}_\mathfrak{p}}(K)(s{h_K}(\xi ) + x \cdot \xi ).\]
Since the centroid of $\mu_\mathfrak{p}(K,\cdot)$ is at the origin,
this implies that $x \cdot \xi = 0$, for
$\mu_\mathfrak{p}(K,\cdot)$-almost all $\xi\in \mathbb{S}^{n-1}$.
Note that the $\mathfrak{p}$-capacitary measure
$\mu_\mathfrak{p}(K,\cdot)$ is not concentrated on any great
subsphere of $\mathbb{S}^{n-1}$. Hence, $x=o$, which in turn implies
that $K$ and ${\scriptstyle \pmb \lbrack} f {\scriptstyle \pmb
\rbrack}$ are dilates.

Conversely, assume that $K$ and ${\scriptstyle \pmb \lbrack} f
{\scriptstyle \pmb \rbrack}$ are dilates, say, $K=s{\scriptstyle
\pmb \lbrack} f {\scriptstyle \pmb \rbrack}$ for some $s>0$. From
our assumption, (\ref{DefOfMixedCapOfFunctions}) combined with the
fact that $\mu_{\mathfrak{p}}(s{\scriptstyle \pmb \lbrack} f
{\scriptstyle \pmb \rbrack},\cdot)=s^{n-{\mathfrak
p}-1}\mu_{\mathfrak{p}}({\scriptstyle \pmb \lbrack} f {\scriptstyle
\pmb \rbrack},\cdot)$, Lemma \ref{Alekidentity}, the definition that
${\rm C}_\mathfrak{p}(f)={\rm C}_\mathfrak{p}({\scriptstyle \pmb
\lbrack} f {\scriptstyle \pmb \rbrack})$, the fact that ${\rm
C}_\mathfrak{p}(s {\scriptstyle \pmb \lbrack} f {\scriptstyle \pmb
\rbrack})=s^{n-\mathfrak{p}}{\rm C}_\mathfrak{p}({\scriptstyle \pmb
\lbrack} f {\scriptstyle \pmb \rbrack})$, and finally our assumption
again, it follows that
\begin{align*}
{{\rm C}_{p,\mathfrak{p}}}(K,f) &= {{\rm
C}_\mathfrak{p}^p}(s{\scriptstyle \pmb \lbrack} f
{\scriptstyle \pmb \rbrack},f) \\
&= {s^{n - \mathfrak{p} - p}}{{\rm C}_\mathfrak{p}^p}({\scriptstyle
\pmb \lbrack} f
{\scriptstyle \pmb \rbrack},f) \\
&= {s^{n - \mathfrak{p} - p}}{{\rm C}_\mathfrak{p}}({\scriptstyle
\pmb \lbrack} f
{\scriptstyle \pmb \rbrack}) \\
&= {s^{n - \mathfrak{p} - p}}{{\rm C}_\mathfrak{p}}{({\scriptstyle
\pmb \lbrack} f
{\scriptstyle \pmb \rbrack})^{\frac{{n - \mathfrak{p} - p}}{{n - \mathfrak{p}}}}}{{\rm C}_\mathfrak{p}}{(f)^{\frac{p}{{n - \mathfrak{p}}}}} \\
&= {{\rm C}_\mathfrak{p}}{(s{\scriptstyle \pmb \lbrack} f
{\scriptstyle \pmb \rbrack})^{\frac{{n - \mathfrak{p} - p}}{{n - \mathfrak{p}}}}}{{\rm C}_\mathfrak{p}}{(f)^{\frac{p}{{n - \mathfrak{p}}}}} \\
&= {{\rm C}_\mathfrak{p}}{(K)^{\frac{{n - \mathfrak{p} - p}}{{n -
\mathfrak{p}}}}}{{\rm C}_\mathfrak{p}}{(f)^{\frac{p}{{n -
\mathfrak{p}}}}}.
\end{align*}
This completes the proof.
\end{proof}

From Theorem 3.8, we have that  for any $L\in\mathcal{K}^n_o$,
\[ {\rm C}_{p,\mathfrak{p}}(K,L)^{n - p}  \ge {{\rm
C}_\mathfrak{p}}{(K)^{n - \mathfrak{p} - p}}{{\rm C}_p\mathfrak{}}{(L)^p}, \]
with equality if and only if $K$ and $L$ are dilates.

The next result is an $L_p$ extension of the
$\mathfrak{p}$-capacitary isoperimetric inequality on  the total
mass of the measure $\mu_{p,\mathfrak{p}}(K,\cdot)$,

\begin{corollary}
Suppose  $1\leq p<\infty$ and  $1<\mathfrak{p}<n$. If $K\in\mathcal{K}^n_o$, then
\[{\mu_{p,\mathfrak{p}}}{(K,{\mathbb{S}^{n - 1}})^{n - \mathfrak{p}}} \ge {n^p}{\omega _n}^p{\left( {\frac{{n - \mathfrak{p}}}{{\mathfrak{p} - 1}}} \right)^{(\mathfrak{p} - 1)p}}{{\rm C}_\mathfrak{p}}{(K)^{n - \mathfrak{p} - p}},\]
with equality if and only if $K$ is an origin-symmetric ball.
\end{corollary}

\begin{proof}
Let $L$ be  the unit ball $B$ in $\mathbb{R}^n$.  Since ${\rm
C}_\mathfrak{p}(B)=n\omega_n\left(\frac{n-\mathfrak{p}}{\mathfrak{p}-1}\right)^{\mathfrak{p}-1}$,  from the  $L_p$ capacitary Minkowski
inequality,
the desired inequality with its equality condition is obtained.
\end{proof}

Let $f_1,f_2,g\in C_+(\mathbb{S}^{n-1})$. From  the definition of $f_1+_p
f_2$ and (\ref{DefOfMixedCapOfFunctions}), it follows that
\[{\rm C}_{p,\mathfrak{p}}({\scriptstyle \pmb \lbrack} g
{\scriptstyle \pmb \rbrack},f_1+_p f_2)={\rm
C}_{p,\mathfrak{p}}({\scriptstyle \pmb \lbrack} g {\scriptstyle \pmb
\rbrack},f_1)+{\rm C}_{p,\mathfrak{p}}({\scriptstyle \pmb \lbrack} g
{\scriptstyle \pmb \rbrack},f_2). \] This, combined with Theorem
\ref{p-capMinkowskiineqthm}, yields the inequality
\[{{\rm C}_{p,\mathfrak{p}}}({\scriptstyle \pmb \lbrack} g
{\scriptstyle \pmb \rbrack}, f_1+_p f_2) \ge {{\rm
C}_\mathfrak{p}}{({\scriptstyle \pmb \lbrack} g {\scriptstyle \pmb
\rbrack})^{\frac{{n - \mathfrak{p} - p}}{{n - \mathfrak{p}}}}}\left(
{{{\rm C}_\mathfrak{p}}{{(f_1)}^{\frac{p}{{n - \mathfrak{p}}}}} +
{{\rm C}_\mathfrak{p}}{{(f_2)}^{\frac{p}{{n - \mathfrak{p}}}}}}
\right),\] with equality if and only if ${\scriptstyle \pmb \lbrack}
f_1 {\scriptstyle \pmb \rbrack}$ and ${\scriptstyle \pmb \lbrack}
f_2 {\scriptstyle \pmb \rbrack}$ are dilates of ${\scriptstyle \pmb
\lbrack} g {\scriptstyle \pmb \rbrack}$. Hence, let $g=f_1+_p f_2$,
it yields  an $L_p$ extension of the Colesanti-Salani
Brunn-Minkowski inequality.

\begin{theorem}\label{capacitaryBM1}
Suppose $1\leq p<\infty$ and $1<\mathfrak{p}<n$. If $f_1,f_2\in
C_+(\mathbb{S}^{n-1})$,   then
\[{{\rm C}_\mathfrak{p}}{({f_1}{ + _p}{f_2})^{\frac{p}{{n -
\mathfrak{p}}}}} \ge {{\rm C}_\mathfrak{p}}{({f_1})^{\frac{p}{{n -
\mathfrak{p}}}}} + {{\rm C}_\mathfrak{p}}{({f_2})^{\frac{p}{{n -
\mathfrak{p}}}}},\] with equality if and only if ${\scriptstyle \pmb
\lbrack} f_1 {\scriptstyle \pmb \rbrack}$ and ${\scriptstyle \pmb
\lbrack} f_2 {\scriptstyle \pmb \rbrack}$ are dilates.
\end{theorem}

Consequently, for any $K,L\in\mathcal{K}^n_o$,
\begin{equation}\label{LqCapBMineq}
{{\rm C}_\mathfrak{p}}{(K{ + _p}L)^{\frac{p}{{n - \mathfrak{p}}}}}
\ge {{\rm C}_\mathfrak{p}}{(K)^{\frac{p}{{n - \mathfrak{p}}}}} +
{{\rm C}_\mathfrak{p}}{(L)^{\frac{p}{{n - \mathfrak{p}}}}},
\end{equation}
with equality if and only if $K$ and $L$ are dilates.

\begin{remark}
The $\mathfrak{p}$-capacitary Brunn-Minkowski inequality also yields
the $\mathfrak{p}$-capacitary Minkowski inequality. Indeed, consider
the nonnegative  concave function
\[f(t)={\rm C}_\mathfrak{p}(K+_p t\cdot L)^\frac{p}{n-\mathfrak{p}}-{\rm
C}_\mathfrak{p}(K)^\frac{p}{n-\mathfrak{p}}-t{\rm
C}_\mathfrak{p}(L)^\frac{p}{n-\mathfrak{p}}.
\]
The $\mathfrak{p}$-capacitary Brunn-Minkowski inequality and
Corollary \ref{LqVariationalFormulaOfCap2} yield
\[\mathop {\lim }\limits_{t \to {0^ + }} \frac{{f(t) - f(0)}}{t} = {{\rm C}_\mathfrak{p}}{(K)^{\frac{p}{{n - \mathfrak{p}}} - 1}}{{\rm C}_{p,\mathfrak{p}}}(K,L) - {{\rm C}_\mathfrak{p}}{(L)^{\frac{p}{{n - \mathfrak{p}}}}} \ge 0.\]
By the equality condition of $\mathfrak{p}$-capacitary
Brunn-Minkowski, if equality holds on the right, the function $f$
must be linear and thus $K,L$ must be dilates.
\end{remark}

\begin{remark}
Suppose that $K,L\in\mathcal{K}^n_o$, $1\leq p<\infty$ and  $1<\mathfrak{p}<n$.
Let $0<s<1$. From the $(n-\mathfrak{p})$-ordered positive
homogeneity of ${\rm C}_\mathfrak{p}$ and the definition of $L_p$
scalar multiplication, the inequality (\ref{LqCapBMineq}) has the
following equivalent forms:

\begin{enumerate}
\item ${{\rm C}_\mathfrak{p}}{((1-s)\cdot K{ + _p} s\cdot L)^{\frac{p}{{n - \mathfrak{p}}}}} \ge (1-s)
{{\rm C}_\mathfrak{p}}{(K)^{\frac{p}{{n - \mathfrak{p}}}}} + s {{\rm
C}_\mathfrak{p}}{(L)^{\frac{p}{{n - \mathfrak{p}}}}}$.
\item ${{\rm C}_\mathfrak{p}}\left( {(1 - s) \cdot K{ + _p}s \cdot L} \right) \ge {{\rm C}_\mathfrak{p}}{(K)^{1 - s}}{{\rm C}_\mathfrak{p}}{(L)^s}$.
\item ${{\rm C}_\mathfrak{p}}\left( {(1 - s) \cdot K{ + _p}s \cdot L} \right)\ge \min \left\{ {{{\rm C}_\mathfrak{p}}(K),{{\rm C}_\mathfrak{p}}(L)} \right\}$.
\item If
${\rm C}_\mathfrak{p}(K)={\rm C}_\mathfrak{p}(L)=1$, then $ {\rm C}_\mathfrak{p}((1-s)\cdot K+_p s\cdot L)\ge 1$.
\end{enumerate}

Recall that $K+_\infty L={\rm conv}(K\cup L)$. From the monotonicity of ${\rm C}_\mathfrak{p}$, it  yields that
\[{\rm C}_\mathfrak{p}(K+_\infty L)\ge\max\left\{ {\rm C}_\mathfrak{p}(K), {\rm C}_\mathfrak{p}(L) \right\}. \]
In fact, from the continuity of $K+_p L$ in $p$ and the
continuity of ${\rm C}_\mathfrak{p}$ on $\mathcal{K}^n_o$, the
inequality (\ref{LqCapBMineq}) will become the above, as
$p\to\infty$.
\end{remark}

\subsection{Uniqueness of the $\boldsymbol L_p$ $\mathfrak p$-capacitary measures}

In this part, we show an immediate application of the $L_p$
Minkowski inequality for $\mathfrak{p}$-capacity to the uniqueness
of the $L_p$ Minkowski problem for $\mathfrak{p}$-capacity, which is
closely related with the following question:

\emph{If $K,L\in\mathcal{K}^n_o$ are such that
$\mu_{p,\mathfrak{p}}(K,\cdot)=\mu_{p,\mathfrak{p}}(L,\cdot)$, then
is this the case that $K=L$?}

Theorems \ref{characterizaion1} (2) and \ref{characterizaion2} (2)
affirm this question. In fact, we  show a series of
characterizations for identity of convex bodies.

\begin{theorem}\label{characterizaion1}
Suppose that $K,L\in\mathcal{K}^n_o$ and $\mathcal{C}$ is a subset
of $\mathcal{K}^n_o$ such that $K,L\in\mathcal{C}$. Let
$1<p<\infty$,  $1<\mathfrak{p}<n$ and $n-\mathfrak{p}\neq p$.   Then
the following assertions hold.

\noindent (1) If ${\rm C}_{p,\mathfrak{p}}(K,Q)={\rm
C}_{p,\mathfrak{p}}(L,Q)$ for all $Q\in\mathcal{C}$, then $K=L$.

\noindent (2) If
$\mu_{p,{\mathfrak{p}}}(K,\cdot)=\mu_{p,\mathfrak{p}}(L,\cdot)$,
then $K=L$.

\noindent (3) If ${\rm C}_{p,\mathfrak{p}}(Q,K)={\rm
C}_{p,\mathfrak{p}}(Q,L)$ for all $Q\in\mathcal{C}$, then $K=L$.
\end{theorem}

\begin{proof}
Since ${\rm C}_{p,\mathfrak{p}}(K,K)={\rm C}_\mathfrak{p}(K)$, it
follows that $ {{\rm C}_{p,\mathfrak{p}}}(L,K)={{\rm
C}_\mathfrak{p}}(K) $  by the assumption. By the
$\mathfrak{p}$-capacitary Minkowski  inequality ${{\rm
C}_{p,\mathfrak{p}}}(L,K) \ge {\rm
C}_\mathfrak{p}(L)^{(n-\mathfrak{p}-p)/(n-\mathfrak{p})} {\rm
C}_\mathfrak{p}(K)^{p/(n-\mathfrak{p})}$, we have
\[{{\rm C}_\mathfrak{p}}{(K)^{\frac{{n - \mathfrak{p} - p}}{{n -
\mathfrak{p}}}}} \ge {{\rm C}_\mathfrak{p}}{(L)^{\frac{{n -
\mathfrak{p} - p}}{{n - \mathfrak{p}}}}},\] with equality if and
only if $K$ and $L$ are dilates. This inequality is reversed if
interchanging $K$ and $L$. So, $ {{\rm C}_\mathfrak{p}}(K)={{\rm
C}_\mathfrak{p}}(L) $, and  $K$ and $L$ are dilates. Assume that
$K=s L$, for some $s>0$. Since ${\rm
C}_\mathfrak{p}(sL)=s^{n-\mathfrak{p}}{\rm C}_\mathfrak{p}(L)$, it
follows that $s=1$. Thus, $K=L$.

If $\mu_{p,\mathfrak{p}}(K,\cdot)=\mu_{p,\mathfrak{p}}(L,\cdot)$,
then ${\rm C}_{p,\mathfrak{p}}(K,Q)={\rm C}_{p,\mathfrak{p}}(L,Q)$
for any $Q\in\mathcal{K}^n_o$. Thus, $K=L$ by (1).

(3) can be proved by the similar arguments in (1).
\end{proof}

If $p=1$ in Theorem \ref{characterizaion1}, then $K$ and $L$ are
translates each other.

\begin{theorem}\label{characterizaion4}
Suppose that $K,L\in\mathcal{K}^n_o$ are such that $\mu_{p,
\mathfrak{p}}(K,\cdot)\le \mu_{p, \mathfrak{p}}(L,\cdot)$. Let
$1<p<\infty$, $1<\mathfrak{p}<n$ and $n-\mathfrak{p}\neq p$. Then
the following assertions hold.

\noindent (1) If ${\rm C}_\mathfrak{p}(K)\ge {\rm
C}_\mathfrak{p}(L)$ and $p<n-\mathfrak{p}$, then $K=L$.

\noindent (2) If ${\rm C}_\mathfrak{p}(K)\le {\rm
C}_\mathfrak{p}(L)$ and $p>n-\mathfrak{p}$, then $K=L$.
\end{theorem}

\begin{proof}
From ${\rm C}_{p,\mathfrak{p}}(L,L)={\rm C}_\mathfrak{p}(L)$,
together with the assumption $\mu_{p,\mathfrak{p}}(K,\cdot)\le
\mu_{p,\mathfrak{p}}(L,\cdot)$ and Definition \ref{DefLqMixedCap},
the $\mathfrak{p}$-capacitary Minkowski inequality, and the
assumptions in (1) or (2), we have
\begin{align*}
{{\rm C}_\mathfrak{p}}(L)
&\ge {{\rm C}_{p,\mathfrak{p}}}(K,L) \\
&\ge {{\rm C}_\mathfrak{p}}{(K)^{\frac{{n - \mathfrak{p} - p}}{{n - \mathfrak{p}}}}}{{\rm C}_\mathfrak{p}}{(L)^{\frac{p}{{n - \mathfrak{p}}}}} \\
&\ge {{\rm C}_\mathfrak{p}}{(L)^{\frac{{n - \mathfrak{p} - p}}{{n - \mathfrak{p}}}}}{{\rm C}_\mathfrak{p}}{(L)^{\frac{p}{{n - \mathfrak{p}}}}} \\
&= {{\rm C}_\mathfrak{p}}(L).
\end{align*}
Thus, ${\rm C}_\mathfrak{p}(K)={\rm C}_\mathfrak{p}(L)$, and $K$ and
$L$ are dilates.  Hence, $K=L$.
\end{proof}

When $n-\mathfrak{p}= p$, we have the following result.

\begin{theorem}\label{characterizaion2}
Suppose that $K,L\in\mathcal{K}^n_o$  and $\mathcal{C}$ is a subset
of $\mathcal{K}^n_o$ such that $K,L\in\mathcal{C}$. Let
$1<p<\infty$ and  $1<\mathfrak{p}<n$.  Then the following assertions
hold.

\noindent (1) If $ {\rm C}_{n-\mathfrak{p},\mathfrak{p}}(K,Q)\ge
{\rm C}_{n-\mathfrak{p}, \mathfrak{p}}(L,Q)$  for all
$Q\in\mathcal{C}$, then $K$ and $L$ are dilates.

\noindent (2) If $ \mu_{n-\mathfrak{p},\mathfrak{p}}(K,\cdot)\ge
\mu_{n-\mathfrak{p},\mathfrak{p}}(L,\cdot)$, then $K$ and $L$ are
dilates. Therefore, $\mu_{n-\mathfrak{p}, \mathfrak{p}}
(K,\cdot)=\mu_{n-\mathfrak{p}, \mathfrak{p}}(L,\cdot)$.
\end{theorem}

\begin{proof}
Take $Q=K$. From the fact ${\rm
C}_{n-\mathfrak{p},\mathfrak{p}}(K,K)={\rm C}_\mathfrak{p}(K)$, the
assumption in (1) and the $\mathfrak{p}$-capacitary Minkowski
inequality, we have
\[{\rm C}_\mathfrak{p}(K) \ge {{\rm C}_{n-\mathfrak{p},\mathfrak{p}} }(L,K) \ge {\rm C}_\mathfrak{p}(K).\]
Thus, all the equalities in the above hold and $K$ and $L$ are
dilates by the equality condition of the $\mathfrak{p}$-capacitary
Minkowski inequality. Incidentally,  we obtain
$\mu_{n-\mathfrak{p},\mathfrak{p}}(K,\cdot)=\mu_{n-\mathfrak{p},\mathfrak{p}}(L,\cdot)$
by Lemma \ref{positivehomogeneitymeasure2}. With (1) in hand, (2)
can be derived directly.
\end{proof}

\vskip 25pt
\section{\bf The discrete $\boldsymbol{L_p}$ Minkowski problem for $\boldsymbol{\mathfrak{p}}$-capacity}
\vskip 10pt

Throughout this section, let $1< p<\infty$ and $1<\mathfrak{p}<n$.
Suppose that $\xi_1,\ldots,\xi_m\in\mathbb{S}^{n-1}$ are pairwise distinct
and not contained in a closed hemisphere, and
$c_1,\ldots,c_m$ are positive numbers. Denote by $\delta_{\xi_i}$
the probability measure with unit point mass at $\xi_i$. We focus on the following.

\vskip 8pt

\noindent \textbf{Problem 1.}  Among all polytopes in $\mathbb{R}^n$
with the origin in their interiors,   find a polytope $P$ such that
$\frac{\mu_{p,\mathfrak{p}}(P,\cdot)}{{\rm C}_\mathfrak{p}(P)}=
\sum_{i=1}^mc_i\delta_{\xi_i}$.

\vskip 8pt
We present a solution to Problem 1.

\begin{theorem}\label{discreteCapMinkowskiThm}
Suppose $1< p<\infty$ and $1<\mathfrak{p}<n$. If
$\xi_1,\ldots,\xi_m\in\mathbb{S}^{n-1}$ are pairwise distinct which are
 not contained in any closed hemisphere, and $c_1,\ldots,c_m$ are
positive numbers,  then there exists a unique convex polytope $P\in\mathcal{K}^n_o$ such that
\[\frac{\mu_{p,\mathfrak{p}}(P,\cdot)}{{\rm C}_\mathfrak{p}(P)} =  \sum_{i=1}^mc_i\delta_{\xi_i}. \]
\end{theorem}

To prove this theorem, we need  to make some preparations. Let
$\mathbb{R}^m_*=[0,\infty)^m$. For each nonzero
$y=(y_1,\ldots,y_m)\in\mathbb{R}^m_*$, define
\[P(y)=\bigcap_{i=1}^m\left\{x\in\mathbb{R}^n, x\cdot\xi_i\le y_i \right\}.\]
Then the unit outer normals to facets of $P(y)$ belong to
$\{\xi_1,\ldots,\xi_m\}$, and $P(y)$ is a polytope containing $o$.
Since  $\mu_\mathfrak{p}(P(y),\cdot)$ is absolutely continuous with
respect to $S_{P(y)}$,  we have
\begin{equation}\label{mixedcapacityofAlekpoytope1}
{\rm C}_\mathfrak{p}(P(y),P(z))=\frac{\mathfrak{p} - 1}{n -
\mathfrak{p}}\sum_{i = 1}^m h_{P(z)}(\xi_i){\mu
_\mathfrak{p}}(P(y),\{\xi _i\} ).
\end{equation}
Since  $h_{P(y)}(\xi_i)\le y_i$, with equality if
$S_{P(y)}(\{\xi_i\})>0$, for $i=1,\ldots,m$, we have
\begin{equation}\label{mixedcapacityofAlekpoytope2}
{{\rm C}_\mathfrak{p}}(P(y)) = \frac{{\mathfrak{p} - 1}}{{n -
\mathfrak{p}}}\sum_{i = 1}^m {{y_i}{\mu _\mathfrak{p}}(P(y),\{ {\xi_i}\} )}.
\end{equation}

To solve Problem 1, our strategy is  to attack the  following Problem 2. In the proof of Theorem 5.1, we can see that Problem 1 is essentially solved once we  solve
Problem 2. Precisely, we show that Problem 1 and Problem 2 have the identical solution.
\vskip 8pt

\noindent \textbf{Problem 2.} Among all elements $y$ in
$\mathbb{R}^m_*$, find an element which solves the following
constrained maximization problem
\[\max\limits_y {\rm C}_\mathfrak{p}(P(y))\quad\quad {\rm subject\; to}\quad \frac{\mathfrak{p}-1}{n-\mathfrak{p}}\sum_{i=1}^m c_iy_i^p=1. \]

\begin{lemma}\label{keyLem1}
${\rm C}_\mathfrak{p}(P(y))$ is continuous with respect to
$y\in\mathbb{R}^m_*\setminus\{o\}$.
\end{lemma}

\begin{proof}
By Aleksandrov's convergence theorem, $P(y)$ is continuous with
respect to $y\in\mathbb{R}^m_*\setminus\{o\}$. So, by the continuity
of $\mathfrak{p}$-capacity with respect to the Hausdorff metric,
${\rm C}_\mathfrak{p}(P(y))$ is continuous with respect to
$y\in\mathbb{R}^m_*\setminus\{o\}$.
\end{proof}

\begin{lemma}\label{keyLem2}
$P\left(\frac{y'+y''}{2} \right)\supseteq
\frac{1}{2}P(y')+\frac{1}{2}P(y'')$, for any nonzero $y',y''\in
\mathbb{R}^m_*$.
\end{lemma}

\begin{proof}
Let $x\in \frac{1}{2}P(y')+\frac{1}{2}P(y'')$. Then there exist $x'\in
P(y')$ and $x''\in P(y'')$, such that $x=\frac{x'+x''}{2}$ and for
each $i$,
\[x'\cdot\xi_i\le y_i'\quad {\rm and}\quad x''\cdot\xi_i\le y_i''.\]
Thus for each $i$, we have
\[x\cdot\xi_i=\frac{x'+x''}{2}\cdot\xi_i\le \frac{y_i'+y_i''}{2},\]
which implies that $x\in P\left(\frac{y'+y''}{2}\right)$.
\end{proof}

To prove Lemma \ref{keyLem3},  we adopt the elegant deformation technique,  which was previously employed
by Hug and LYZ \cite{Minkproblem(HugLYZ)}.
\begin{lemma}\label{keyLem3}
If $y\in \mathbb{R}^m_*$ solves Problem 2, then $o\in {\rm int} P(y)$.
\end{lemma}

\begin{proof}
We argue by contradiction and assume that $o\in \partial P(y)$. Let
$y=(y_1,\ldots,y_m)$ and $h_i=h_{P(y)}(\xi_i)$, for
$i=1,\ldots,m$. Since $o\in\partial P(y)$, w.l.f.g.,
assume that $h_1=\cdots=h_k=0$ and
$h_{k+1},\ldots,h_m>0$, for some $1\le k< m$. In the next, we
will construct a new polytope $P(z)$ with $o$ in its interior,
such that $z$ satisfies the constraint in Problem 2 but ${\rm C}_\mathfrak{p}(P(z))> {\rm C}_\mathfrak{p}(P(y))$.

Let $c= {\sum_{i=1}^k c_i}/{\sum_{i=k+1}^m c_i}$ and $ {t_0} = {
{\min \left\{ {h_i^p/c:1 \le i \le k} \right\}} ^{\frac{1}{p}}}$.
For $0\le t<t_0$, let
\[{y_t} = \left( {t, \ldots ,t,{{(h_{k + 1}^p - c{t^p})}^{\frac{1}{p}}}, \ldots ,{{(h_m^p - c{t^p})}^{\frac{1}{p}}}} \right).\]
Then, $y_t\in (0,\infty)^m$ for $0<t<t_0$ and $P(y_0)=P(y)$.
From (\ref{mixedcapacityofAlekpoytope1}) combined with
(\ref{mixedcapacityofAlekpoytope2}),  and then the fact  $\lim_{t\to
0^+}P(y_t)= P(y)$ combined with the weak convergence of
$\mathfrak{p}$-capacitary measures, we have
\begin{align*}
&\mathop {\lim }\limits_{t \to {0^ + }} \frac{{{{\rm{C}}_\mathfrak{p}}(P({y_t})) - {{\rm{C}}_\mathfrak{p}}(P({y_t}),P(y))}}{t} \\
&\quad= \frac{{\mathfrak{p} - 1}}{{n - \mathfrak{p}}}\left( {\sum\limits_{i = 1}^k {\mathop {\lim }\limits_{t \to {0^ + }} \frac{{t - 0}}{t}{\mu _\mathfrak{p}}(P({y_t}),\{ {\xi _i}\} )}  + \sum\limits_{i = k + 1}^m {\mathop {\lim }\limits_{t \to {0^ + }} \frac{{{{(h_i^p - c{t^p})}^{\frac{1}{p}}} - {h_i}}}{t}{\mu _\mathfrak{p}}(P({y_t}),\{ {\xi _i}\} )} } \right) \\
&\quad= \frac{{\mathfrak{p} - 1}}{{n - \mathfrak{p}}}\sum\limits_{i = 1}^k {{\mu _\mathfrak{p}}(P(y),\{ {\xi _i}\})} .
\end{align*}

Since there is at least one facet of $P(y)$ containing $o$, it follows that
$\sum_{i=1}^kS_{P(y)}(\{\xi_i\})>0$. Also, by CNSXYZ
\cite[Lemma 2.18]{capacitaryBMtheory(CNSXYZ)}, there exists a positive
constant $c$ depending on $n$, $\mathfrak p$ and the radius of a
ball containing $P(y)$, such that $\mu_\mathfrak{p}(P(y),\cdot)\ge
c^{-\mathfrak{p}} S_{P(y)}$. So,
$\sum_{i=1}^k\mu_\mathfrak{p}(P(y),\{\xi_i\})>0$. This in turn  implies that
\[\mathop {\lim }\limits_{t \to {0^ + }} \frac{{{{\rm C}_\mathfrak{p}}(P({y_t})) - {{\rm C}_\mathfrak{p}}(P({y_t}),P(y))}}{t} > 0.\]

Hence, by the $\mathfrak{p}$-capacitary Minkowski inequality and
continuity of ${\rm C}_\mathfrak{p}(P(y_t))$ in
$t$, we have
\begin{align*}
&{{\rm C}_\mathfrak{p}}{(P(y))^{\frac{{n - \mathfrak{p} - 1}}{{n - \mathfrak{p}}}}}\mathop {\lim \inf }\limits_{t \to {0^ + }} \frac{{{{\rm C}_\mathfrak{p}}{{(P({y_t}))}^{\frac{1}{{n - \mathfrak{p}}}}} - {{\rm C}_\mathfrak{p}}{{(P(y))}^{\frac{1}{{n - \mathfrak{p}}}}}}}{t} \\
&\quad= \mathop {\lim \inf }\limits_{t \to {0^ + }} \frac{{{{\rm C}_\mathfrak{p}}(P({y_t})) - {{\rm C}_\mathfrak{p}}{{(P({y_y}))}^{\frac{{n - \mathfrak{p} - 1}}{{n - \mathfrak{p}}}}}{{\rm C}_\mathfrak{p}}{{(P(y))}^{\frac{1}{{n - \mathfrak{p}}}}}}}{t} \\
&\quad\ge \mathop {\lim \inf }\limits_{t \to {0^ + }} \frac{{{{\rm C}_\mathfrak{p}}(P({y_t})) - {{\rm C}_\mathfrak{p}}(P({y_y}),P(y))}}{t} \\
&\quad > 0.
\end{align*}
Consequently, for  sufficiently small $t$, we have
${{\rm C}_\mathfrak{p}}(P({y_t})) > {{\rm C}_\mathfrak{p}}(P(y))$.

Now, choose a sufficiently small $t>0$ and let
\[z = \left( {{{({y_1}^p + {t^p})}^{\frac{1}{p}}}, \cdots ,{{({y_k}^p + {t^p})}^{\frac{1}{p}}},{{(y_{k + 1}^p - c{t^p})}^{\frac{1}{p}}}, \ldots ,{{(y_m^p - c{t^p})}^{\frac{1}{p}}}} \right).\]
Then $z$ satisfies the constraint in Problem 2. Since $0<h_i\le
y_i$, $k+1\le i\le m$, it follows that $P(y_t)\subseteq P(z)$. So, ${\rm
C}_\mathfrak{p}(P(z))> {{\rm C}_\mathfrak{p}}(P(y))$. In light of $o\in {\rm int} P(y_t)$, it yields that $o\in {\rm int} P(z)$.
\end{proof}

Let $y=(y_1,\ldots,y_m)\in \mathbb{R}^m_+=(0, +\infty)^{m}$. For $z\in\mathbb{R}^m$, applying the Hadamard variational
formula to $P(y+t z)$, it yields that
\[{\left. {\frac{d{{\rm C}_\mathfrak{p}}(P(y + tz))}{{dt}}} \right|_{t = 0}} = (\mathfrak{p} - 1)\sum\limits_{i = 1}^m {{z_i}{\mu _\mathfrak{p}}(P(y),\{ {\xi _i}\} )} .\]
Thus, we obtain the following useful formula.

\begin{lemma}\label{keyLem4}
$\frac{\partial {{\rm C}_\mathfrak{p}}(P(y))}{{\partial {y_i}}} =
(\mathfrak{p} - 1){\mu _\mathfrak{p}}(P(y),\{ {\xi _i}\} )$,  for
$y\in\mathbb{R}^m_+$ and $ i=1,\cdots,m$.
\end{lemma}

\begin{lemma}\label{keylem5}
Suppose $1<p<\infty$ and $1<{\mathfrak p}<n$. If
$K,L\in\mathcal{K}^n_o$ are such that ${\rm C}_{\mathfrak
p}(K)^{-1}\mu_{p,\mathfrak p}(K,\cdot)={\rm C}_{\mathfrak
p}(L)^{-1}\mu_{p,\mathfrak p}(L,\cdot)$, then $K=L$.
\end{lemma}

\begin{proof}
From the Poincar$\rm\acute{e}$ $\mathfrak p$-capacity formula
together with Definition \ref{DefOfCapacitaryMeasure},   the
supposition that ${\rm C}_{\mathfrak p}(K)^{-1}\mu_{p,\mathfrak
p}(K,\cdot)={\rm C}_{\mathfrak p}(L)^{-1}\mu_{p,\mathfrak
p}(L,\cdot)$, Definitions \ref{DefLqMixedCap} and
\ref{DefOfCapacitaryMeasure}, and Theorem
\ref{p-capMinkowskiineqthm}, it follows that
\begin{align*}
 1 &= \frac{{\mathfrak{p} - 1}}{{(n - \mathfrak{p}){{\rm{C}}_\mathfrak{p}}(L)}}\int\limits_{{S^{n - 1}}} {h_L^pd{\mu _{p,\mathfrak{p}}}(L, \cdot )}  \\
  &= \frac{{\mathfrak{p} - 1}}{{(n - \mathfrak{p}){{\rm{C}}_\mathfrak{p}}(K)}}\int\limits_{{S^{n - 1}}} {h_L^pd{\mu _{p,\mathfrak{p}}}(K, \cdot )}  \\
  &= \frac{{{C_{p,\mathfrak{p}}}(K,L)}}{{{{\rm{C}}_\mathfrak{p}}(K)}} \\
  &\ge {\left( {\frac{{{{\rm{C}}_\mathfrak{p}}(L)}}{{{{\rm{C}}_\mathfrak{p}}(K)}}} \right)^{\frac{p}{{n - \mathfrak{p}}}}}.
 \end{align*}
Thus, ${\rm C}_{\mathfrak{p}}(K) \ge {\rm C}_{\mathfrak{p}}(L)$.
Interchanging $K$ and $L$, we have ${\rm C}_{\mathfrak{p}}(L) \ge
{\rm C}_{\mathfrak{p}}(K)$. So, by Theorem
\ref{p-capMinkowskiineqthm}, the convex bodies $K$ and $L$ are
dilates, so that ${\rm C}_{\mathfrak{p}}(K) = {\rm
C}_{\mathfrak{p}}(L)$. In other words, $K=L$.
\end{proof}

What follows provides the proof of  Theorem \ref{discreteCapMinkowskiThm}.

\begin{proof}[\bf Proof of Theorem \ref{discreteCapMinkowskiThm}]
Let
\[\mathcal{B}=\left\{y\in\mathbb{R}^m_*: \frac{\mathfrak{p}-1}{n-\mathfrak{p}}\sum_{i=1}^m c_iy_i^p \le 1 \right\} \]
and
\[\mathcal{E}_t=\left\{y\in\mathbb{R}^m_*: {\rm C}_\mathfrak{p}(P(y))\ge t \right\},\quad {\rm for}\; t>0. \]
Then $\mathcal{B}$ is a convex body in $\mathbb{R}^m$. By Lemma \ref{keyLem1}, $\mathcal{E}_t$ is a closed set.

Pick up  $y',y''\in \mathcal{E}_t$. From Lemma \ref{keyLem2}, the
monotonicity of $\mathfrak{p}$-capacity  and the $\mathfrak{p}$-capacitary Brunn-Minkowski
inequality, it follows that
\begin{align*}
{{\rm C}_\mathfrak{p}}\left( {P\left( {\frac{{y' + y''}}{2}} \right)} \right) &\ge {{\rm C}_\mathfrak{p}}\left( {\frac{1}{2}P(y') + \frac{1}{2}P(y'')} \right) \\
&\ge {\left( {\frac{1}{2}{{\rm C}_\mathfrak{p}}{{\left( {P(y')} \right)}^{\frac{1}{{n - \mathfrak{p}}}}} + \frac{1}{2}{{\rm C}_\mathfrak{p}}{{\left( {P(y')} \right)}^{\frac{1}{{n - \mathfrak{p}}}}}} \right)^{n - \mathfrak{p}}} \\
&= t,
\end{align*}
which implies that $\frac{y'+y''}{2}\in \mathcal{E}_t$. Hence,
$\mathcal{E}_t$ is convex. Since ${\rm C}_\mathfrak{p}(P(sy))=s^{n-\mathfrak{p}}{\rm C}_\mathfrak{p}(P(y))$,
for nonzero $y\in\mathbb{R}^m_*$ and $s>0$, it follows that
$\mathcal{E}_t$ is unbounded and strictly decreasing (with respect
to set inclusion) when $t$ is increasing, and its interior is
nonempty. So, when $t$ is sufficiently big,
$\mathcal{E}_t\cap\mathcal{B}=\emptyset$; when $t$ is sufficiently
small, ${\rm int}(\mathcal{E}_t)\cap {\rm
int}(\mathcal{B})\neq\emptyset$.

Consequently,  there exists a unique $t_0>0$ such that
$ \mathcal{E}_{t_0}\cap\mathcal{B}=\partial \mathcal{E}_{t_0}\cap \partial\mathcal{B}$.
Since the set $\{y\in\mathbb{R}^m: \frac{\mathfrak{p}-1}{n-\mathfrak{p}}\sum_{i=1}^mc_i|y_i|^p\le 1\}$
is a strictly convex body in $\mathbb{R}^m$ with smooth boundary,
the sets $ \mathcal{E}_{t_0}$ and $ \mathcal{B}$ necessarily share a unique
common boundary point, say $\tilde{y}$. In other words, for any $y\in \partial\mathcal{B}$, we have
$${\rm C}_\mathfrak{p}(P(\tilde{y}))\geq {\rm C}_\mathfrak{p}(P(y)),$$
with equality if and only if $y=\tilde{y}$. This proves the unique existence of  solution to Problem 2.

We proceed to prove that $P(\tilde{y})$ uniquely solves Problem 1.

By Lemma \ref{keyLem3}, the polytope $P(\tilde{y})$ contains the
origin in its interior. Therefore, $\tilde{y}\in\mathbb{R}^m_+$.
Since  ${\left. {\nabla \left( {\sum\limits_{i = 1}^m
{{c_i}y_i^p} } \right)} \right|_{\tilde y}}$ is a normal  of
$\mathcal{B}$ at  $\tilde{y}$ with components $p c_i\tilde{y}_i^{p-1}$, and   $\nabla {\rm C}_\mathfrak{p}(P(y))|_{\tilde{y}}$ is a normal  of $\mathcal{E}_{t_0}$
at  $\tilde{y}$ with components $ (\mathfrak{p} - 1){\mu _\mathfrak{p}}(P(\tilde y),\{{\xi _i}\})$  by Lemma \ref{keyLem4},
so there exists a unique
$s_0>0$ such that for each $i$,
${c_i}{\tilde y}_i^{p} = s_0{\tilde y}_i{\mu _\mathfrak{p}}(P(\tilde y),\{ {\xi _i}\} ).$
Since for each $i$, $c_i>0$ and ${\tilde y}_i>0$,  this in turn implies that
$\mu_\mathfrak{p}(P(\tilde{y}),\{\xi_i\})>0$. In light of
$\mu_\mathfrak{p}(P(\tilde{y}),\cdot)$ is absolutely continuous with
respect to $S_{P(\tilde{y})}$, so each $\xi_i$ is a unit normal
of $P(\tilde{y})$. Hence, $h_{P(\tilde{y})}(\xi_i)=y_i$, for each
$i$. Consequently,
\begin{align*}
s_0{{\rm C}_\mathfrak{p}}(P(\tilde{y})) &= s_0 \cdot \frac{{\mathfrak{p} - 1}}{{n - \mathfrak{p}}}\sum\limits_{i = 1}^m {{\tilde{y}_i}{\mu _\mathfrak{p}}(P(\tilde{y}),\{ {\xi _i}\} )}  \\
&= \frac{{\mathfrak{p} - 1}}{{n - \mathfrak{p}}}\sum\limits_{i = 1}^m {s_0{\tilde{y}_i}{\mu _\mathfrak{p}}(P(\tilde{y}),\{ {\xi _i}\} )}  \\
&= \frac{{\mathfrak{p} - 1}}{{n - \mathfrak{p}}}\sum\limits_{i = 1}^m {{c_i}\tilde{y}_i^p}  \\
&= 1,
\end{align*}
which yields that
\[ s_0=\frac{1}{{\rm C}_\mathfrak{p}(P(\tilde{y}))}. \]

Furthermore,
\begin{align*}
\sum\limits_{i = 1}^m {{c_i}{\delta _{{\xi _i}}}}  &= \frac{{\sum_{i = 1}^m {\tilde y_i^{1 - p}{\mu _\mathfrak{p}}(P(\tilde y),\{ {\xi _i}\} ){\delta _{{\xi _i}}}} }}{{{{\rm C}_\mathfrak{p}}(P(\tilde y))}} \\
&= \frac{{\sum_{i = 1}^m {{h_{P(\tilde y)}}{{({\xi _i})}^{1 - p}}{\mu _\mathfrak{p}}(P(\tilde y),\{ {\xi _i}\} ){\delta _{{\xi _i}}}} }}{{{{\rm C}_\mathfrak{p}}(P(\tilde y))}} \\
&= \frac{{{\mu_{p,\mathfrak{p}}}(P(\tilde y), \cdot )}}{{{{\rm
C}_\mathfrak{p}}(P(\tilde y))}}.
\end{align*}
Put it in other words, $P(\tilde{y})$ is a solution to Problem 1,
and is unique by Lemma \ref{keylem5}.
\end{proof}

From Theorem \ref{discreteCapMinkowskiThm}, we immediately obtain the following results.

\begin{corollary}\label{discreteCapMinkowskiThm2}
Suppose $1< p<\infty$,  $1<\mathfrak{p}<n$ and $n-\mathfrak{p}\neq
p$. If $\mu$ is a finite discrete measure on $\mathbb{S}^{n-1}$
which is not concentrated on a closed hemisphere, then there exists
a unique convex polytope $P\in\mathcal{K}^n_o$ such that $
\mu_{p,\mathfrak{p}}(P,\cdot)=\mu$.
\end{corollary}

\begin{proof}
By Theorem \ref{discreteCapMinkowskiThm}, there exists a unique
convex polytope $P^*\in\mathcal{K}^n_o$, such that
$\frac{\mu_{p,\mathfrak{p}}(P^*,\cdot)}{{\rm C}_\mathfrak{p}(P^*)}=
\mu$. Let $P = {{\rm C}_\mathfrak{p}}{(P^*)^{ -\frac{1}{{n -
\mathfrak{p} - p}}}}P^*$. Then,
\[\mu  = \frac{{{\mu _{p,\mathfrak{p}}}\left( {{{\rm C}_\mathfrak{p}}{{(P^*)}^{\frac{1}{{n - \mathfrak{p} - p}}}}P, \cdot } \right)}}{{{{\rm C}_\mathfrak{p}}(P^*)}} = \frac{{{{\rm C}_\mathfrak{p}}(P^*){\mu _{p,\mathfrak{p}}}\left( {P, \cdot } \right)}}{{{{\rm C}_\mathfrak{p}}(P^*)}} = {\mu _{p,\mathfrak{p}}}\left( {P, \cdot } \right),\]
as desired.
\end{proof}

The following lemma shows the solution to the \emph{even} $L_p$
Minkowski problem for $\mathfrak{p}$-capacity is symmetric.

\begin{lemma}\label{characterizaion5}
Suppose $1\le p<\infty$ and $1<\mathfrak{p}<n$. If
$K\in\mathcal{K}^n_o$,  then the following statements are
equivalent.

\noindent (1)  $K$ is origin-symmetric when $p>1$, or centrally
symmetric when $p=1$.

\noindent (2) $\mu_{p, \mathfrak{p}}(K,\cdot)$ is even.

\noindent (3) ${\rm C}_{p,\mathfrak{p}}(K,-Q)={\rm
C}_{p,\mathfrak{p}}(K,Q)$, for all $Q\in\mathcal{K}^n_o$.

\noindent (4) ${\rm C}_{p,\mathfrak{p}}(K,-K)={\rm
C}_\mathfrak{p}(K)$.
\end{lemma}

\begin{proof}
When $p=1$, the implication ``(1) $\Rightarrow$ (2)" is obvious.
When $p>1$, the implication ``(1) $\Rightarrow$ (2)" follows from
the facts that $\mu_\mathfrak{p}(K,\cdot)$ is even, $h_K=h_{-K}$ and
Definition \ref{DefOfLqCapMeasure}.

The implication ``(2) $\Rightarrow$ (3)" follows from Definition
\ref{DefLqMixedCap} and the fact that $h_Q(-\xi)=h_{-Q}(\xi)$ for
all $\xi\in\mathbb{S}^{n-1}$.

The implication ``(3) $\Rightarrow$ (4)" is obvious, since ${\rm
C}_{p,\mathfrak{p}}(K,-K)={\rm C}_{p,\mathfrak{p}}(K,K)={\rm
C}_\mathfrak{p}(K)$.

Assume that ${\rm C}_{p,\mathfrak{p}}(K,-K)={\rm
C}_\mathfrak{p}(K)$. From the $\mathfrak{p}$-capacitary Minkowski
inequality and the fact ${\rm C}_\mathfrak{p}(K)={\rm
C}_\mathfrak{p}(-K)$, it follows that
\[{{\rm C}_\mathfrak{p}}(K) = {{\rm C}_{p,\mathfrak{p}}}(K, - K)
\ge {{\rm C}_\mathfrak{p}}{(K)^{\frac{{n - \mathfrak{p} - p}}{{n -
\mathfrak{p}}}}}{{\rm C}_\mathfrak{p}}{( - K)^{\frac{p}{{n -
\mathfrak{p}}}}} = {{\rm C}_\mathfrak{p}}(-K).
\]
So, $K$ and $-K$ are dilates when $p>1$, or homothetic when $p=1$.
\end{proof}

\begin{corollary}\label{discreteCapMinkowskiThm3}
Suppose $1< p<\infty$,   $1<\mathfrak{p}<n$ and $n-\mathfrak{p}\neq
p$. If $\mu$ is a finite even discrete measure on $\mathbb{S}^{n-1}$
which is not concentrated on any great subsphere, then there exists
a unique origin-symmetric convex polytope $P\in\mathcal{K}^n_o$ such
that $ \mu_{p,\mathfrak{p}}(P,\cdot)=\mu$.
\end{corollary}

\begin{proof}
Since $\mu$ is even and  not concentrated on any great subsphere, it
is not concentrated on any closed hemisphere. By Corollary
\ref{discreteCapMinkowskiThm2}, there exists a unique polytope
$P\in\mathcal{K}^n_o$ such that $
\mu_{p,\mathfrak{p}}(P,\cdot)=\mu$. Since
$\mu_{p,\mathfrak{p}}(P,\cdot)$  is even, it implies that  $P$ is
origin-symmetric  by Lemma \ref{characterizaion5}.
\end{proof}

\vskip 25pt
\section{\bf  Revisiting the discrete Minkowski problem for $\mathfrak{p}$-capacity:  CNSXYZ's  problem}
\vskip 10pt

Let $\mu$ be a finite Borel measure  on the unit sphere ${\mathbb
S}^{n-1}$. Consider the following conditions.

($A_1$)  The measure $\mu$ is not concentrated on any great
subsphere.

($A_2$)  The centroid of $\mu$ is at the origin.

($A_3$)  The measure $\mu$ does not have a pair of antipodal point
masses; that is, i.e., if $\mu(\{\xi\})>0$, then $\mu(\{-\xi\})=0$,
for $\xi\in\mathbb{S}^{n-1}$.

Under these conditions, CNSXYZ \cite[pp.
1570-1572]{capacitaryBMtheory(CNSXYZ)} proved the following
important result.

\noindent \textbf{Theorem A.} \emph{Suppose $1<\mathfrak{p}<2\le n$.
If $\mu$ is a finite Borel measure on $\mathbb{S}^{n-1}$ satisfying
conditions ($A_1$)-($A_3$), then there exists a convex body $K$ in
$\mathbb{R}^n$ such that $\mu_\mathfrak{p}(K,\cdot)=\mu$.}

Conditions ($A_1$) and ($A_2$) are both necessary. They  are exactly
the same sufficient and necessary conditions as in Jerison's
solution to the Minkowski problem for electrostatic capacity
\cite{Minkproblem(Jerison)}, as well as in the Aleksandrov
\cite{Minkproblem(Aleksandrov)} and Fenchel and Jessen's
\cite{Minkproblem(FenchelJessen)} solution to the classical
Minkowski problem for the surface area measure.

CNSXYZ \cite{capacitaryBMtheory(CNSXYZ)} emphasized that ($A_3$) is
instead not a necessary condition. They  pointed out that: It would
be interesting if the assumption ($A_3$) could be removed, and it is
a very interesting \emph {open} problem to naturally extend their
result to the range $2<\mathfrak{p}<n$.

In this part, we  solve CNSXYZ's  problem for discrete measures.

\begin{theorem}
Suppose  $1<\mathfrak{p}<n$. If $\mu$ is a discrete measure on
$\mathbb{S}^{n-1}$ satisfying conditions ($A_1$) and ($A_2$), then
there exists a unique (up to a translation) polytope $P$ such that
\[\frac{\mu_\mathfrak{p}(P,\cdot)}{{\rm C}_\mathfrak{p}(P)}=\mu.\]
If in addition $\mu$ is even, then $P$ is centrally symmetric.
\end{theorem}

\begin{proof}
The argument is similar to the proof of Theorem
\ref{discreteCapMinkowskiThm}, so we have to use the notations and
lemmas provided in Section 5. Represent $\mu$ as the form
$\sum_{i=1}^m c_i\delta_{\xi_i}$, where $c_1,\ldots,c_m>0$, and
$\xi_1,\ldots,\xi_m$ are unit vectors which are not contained on any
great subsphere.

We start with considering the simplex
\[S=\left\{y\in\mathbb{R}^m_*: \frac{{\mathfrak p}-1}{n-{\mathfrak p}}\sum_{i=1}^m c_iy_i=1 \right\}. \]
By Lemma \ref{keyLem1} and the compactness of $S$, the functional
${\rm C}_\mathfrak{p}(P(y))$ can attain its maximum on $S$ at a
point $z$, say $z=(z_1,\ldots,z_m)$.

If $z\notin {\rm relint} S$ (i.e., $z$ is not a relative interior
point of $S$), then at least one $z_i$ is $0$, and therefore $o\in
\partial P(z)$. Choose a nonzero $\Delta z\in\mathbb{R}^m$, such that $o\in {\rm int} (P(z)+\Delta z)$.
Let
\[\tilde{y}=(\tilde{y}_1,\ldots,\tilde{y}_m)=z+(\xi_1\cdot\Delta z, \ldots, \xi_m\cdot \Delta z).\]
Then,
\begin{align*}
P(\tilde y) &= \left\{ {x \in {\mathbb{R}^n}:{\xi _i} \cdot x \le {{\tilde y}_i},\;{\rm for}\;i = 1, \ldots ,m} \right\} \\
&= \left\{ {x \in {\mathbb{R}^n}:{\xi _i} \cdot x \le {z_i} + {\xi _i} \cdot \Delta z,\;{\rm for}\;i = 1, \ldots ,m} \right\} \\
&= \left\{ {x \in {\mathbb{R}^n}:{\xi _i} \cdot x \le {z_i},\;{\rm for}\;i = 1, \ldots ,m} \right\} + \Delta z \\
&= P(z) + \Delta z.
\end{align*}
Since $o\in {\rm int} (P(z)+\Delta z)$, it follows that
\[\tilde{y}_1>0,\ldots,\tilde{y}_m>0.\]
From $z\in S$ and the centroid of $\sum_{i=1}^mc_i\delta_{\xi_i}$ is
at the origin, it follows that
\begin{align*}
\sum\limits_{i = 1}^m {{c_i}{{\tilde y}_i}}  &= \sum\limits_{i = 1}^m {{c_i}({z_i} + {\xi _i} \cdot \Delta z)}  \\
&= \sum\limits_{i = 1}^m {{c_i}{z_i}}  + \left( {\sum\limits_{i = 1}^m {{c_i}{\xi _i}} } \right) \cdot \Delta z \\
&= \frac{{n - \mathfrak{p}}}{{\mathfrak{p} - 1}} + o \cdot \Delta z \\
&= \frac{{n - \mathfrak{p}}}{{\mathfrak{p} - 1}},
\end{align*}
i.e., $\frac{{\mathfrak{p} - 1}}{{n - \mathfrak{p}}}\sum\limits_{i =
1}^m {{c_i}{{\tilde y}_i}}  = 1$, which implies that $\tilde{y}\in
S.$ Hence, ${\rm C}_\mathfrak{p}(P(y))$  attains its maximum on $S$
at a relative interior point $\tilde{y}$.

By Lemma \ref{keyLem4} and the Lagrange multiplier theorem, there
exists a suitable constant  $s$, such that for each $i=1,\ldots,m$,
\[{\left. {\frac{{\partial \left( {\frac{{{{\rm C}_\mathfrak{p}}(P(y))}}{{\mathfrak{p} - 1}} - s\sum\limits_{i = 1}^m {{c_i}{y_i}} } \right)}}{{\partial {y_i}}}} \right|_{y = \tilde y}} = {\mu _\mathfrak{p}}(P(\tilde y),\{ {\xi _i}\} ) - s{c_i} = 0.\]
Since $P(\tilde{y})$ is $n$-dimensional and  all the $\tilde{y}_i$
are positive, there is at least one $i_0$ such that
$S_{P(\tilde{y})}(\{\xi_{i_0}\})>0$. Meanwhile, by CNSXYZ
\cite[Lemma 2.18]{capacitaryBMtheory(CNSXYZ)}, there is a positive
constant $c$ depending on $n$, $\mathfrak p$ and and the radius of a
ball containing $P(\tilde{y})$, such that
$\mu_\mathfrak{p}(P(\tilde{y}),\cdot)\ge c^{-\mathfrak{p}}
S_{P(\tilde{y})}$. So,
$\mu_\mathfrak{p}(P(\tilde{y}),\{\xi_{i_0}\})>0$, which implies that
$s>0$, and therefore $\mu_\mathfrak{p}(P(\tilde{y}),\{\xi_{i}\})>0$
for all $i$. In light of $\mu_\mathfrak{p}(P(\tilde{y}),\cdot)$ is
absolutely continuous with respect to $S_{P(\tilde{y})}$,  it
follows that $S_{P(\tilde{y})}(\{\xi_i\})>0$ for all $i$.  So, each
$\xi_i$ is an outer unit normal to the facet of $P(\tilde{y})$,  and
$h_{P(\tilde{y})}(\xi_i)=\tilde{y}_i$.

Hence,
\begin{align*}
{{\rm C}_\mathfrak{p}}(P(\tilde y)) &= \frac{{\mathfrak{p} - 1}}{{n - \mathfrak{p}}}\sum\limits_{i = 1}^m {{{\tilde y}_i}{\mu _\mathfrak{p}}(P(\tilde y),\{ {\xi _i}\} )}  \\
&= s \cdot \frac{{\mathfrak{p} - 1}}{{n - \mathfrak{p}}}\sum\limits_{i = 1}^m {{{\tilde y}_i}{c_i}}  \\
&= s.
\end{align*}
Therefore,
\[\mu  = \sum\limits_{i = 1}^m {{c_i}{\delta _{{\xi _i}}}}  = {s^{ - 1}}\sum\limits_{i = 1}^m {s{c_i}{\delta _{{\xi _i}}}}  = \frac{{\sum\limits_{i = 1}^m {\mu _\mathfrak{p}}(P(\tilde y),\{ {\xi _i}\} )}}{{{{\rm C}_\mathfrak{p}}(P(\tilde y))}}.\]
Take $P=P(\tilde{y})$. Then $P$  is a desired polytope of this
theorem.

What follows shows the uniqueness. Assume  the polytope $P'$
satisfies ${\rm
C}_\mathfrak{p}(P')^{-1}\mu_\mathfrak{p}(P',\cdot)=\mu$.  We will
show that $P$ and $P'$ differ only by a translation.

From the Poincar$\rm{\acute{e}}$ $\mathfrak{p}$-capacity formula,
the assumptions that $\mu={\rm
C}_\mathfrak{p}(P')^{-1}\mu_\mathfrak{p}(P',\cdot)$ and  $\mu={\rm
C}_\mathfrak{p}(P)^{-1}\mu_\mathfrak{p}(P,\cdot)$, the definition of
mixed $\mathfrak{p}$-capacity, and finally the
$\mathfrak{p}$-capacitary Minkowski inequality, it follows that
\begin{align*}
1 &= \frac{{\frac{{\mathfrak{p} - 1}}{{n - \mathfrak{p}}}\int_{{\mathbb{S}^{n - 1}}} {{h_{P'}}d{\mu _\mathfrak{p}}(P', \cdot )} }}{{{{\rm C}_\mathfrak{p}}(P')}} \\
&= \frac{{\mathfrak{p} - 1}}{{n - \mathfrak{p}}}\int\limits_{{\mathbb{S}^{n - 1}}} {{h_{P'}}d\mu }  \\
&= \frac{{\frac{{\mathfrak{p} - 1}}{{n - \mathfrak{p}}}\int\limits_{{\mathbb{S}^{n - 1}}} {{h_{P'}}d{\mu _\mathfrak{p}}(P, \cdot )} }}{{{{\rm C}_\mathfrak{p}}(P)}} \\
&= \frac{{{{\rm C}_\mathfrak{p}}(P,P')}}{{{{\rm C}_\mathfrak{p}}(P)}} \\
&\ge {\left( {\frac{{{\rm C}_\mathfrak{p}(P')}}{{{\rm
C}_\mathfrak{p}(P)}}} \right)^{\frac{1}{{n - \mathfrak{p}}}}}.
\end{align*}
All the above still hold, if interchanging $P$ and $P'$. So, ${\rm
C}_\mathfrak{p}(P')={\rm C}_\mathfrak{p}(P)$. By the equality
condition of the $\mathfrak{p}$-capacitary Minkowski inequality, $P$
and $P'$ differ only by a translation.

Assume that $\mu$ is even. Since $\mu={\rm
C}_\mathfrak{p}(P)^{-1}\mu_\mathfrak{p}(P,\cdot)$, it follows that
the $\mathfrak{p}$-capacitary measure $\mu_\mathfrak{p}(P,\cdot)$ is
even. By Theorem \ref{characterizaion5}, the polytope $P$ is
centrally symmetric.
\end{proof}

\vskip 25pt
\section{\bf  Two dual extremum problems for $\boldsymbol{\mathfrak{p}}$-capacity}
\vskip 10pt

Throughout this section, let $1<p<\infty$ and $1<\mathfrak{p}<n$.
Suppose that $\mu$ is a finite Borel measure on
$\mathbb{S}^{n-1}$, which is not concentrated on any closed
hemisphere. We focus on the general $L_{p}$ Minkowski problem for $\mathfrak{p}$-capacity.

\vskip 8pt

\noindent \textbf{Problem 3.}  Among all convex bodies $Q$ in
$\mathbb{R}^n$ containing the origin, find a body  to solve the following
constrained maximization problem
\[\sup\limits_Q {\rm C}_\mathfrak{p}(Q)\quad\quad {\rm subject\; to}\quad F_p(Q)=1. \]
Here,
\[F_p(Q)=\frac{\mathfrak{p}-1}{n-\mathfrak{p}}\int\limits_{\mathbb{S}^{n-1}}h_Q^pd\mu.\]

Naturally, we also consider  the dual  problem of Problem 3.

\noindent \textbf{Problem 4.}  Among all convex bodies $Q$ in
$\mathbb{R}^n$ containing the origin, find a body to solve the
following constrained minimization problem
\[\inf\limits_Q F_p(Q)\quad\quad {\rm subject\; to}\quad {\rm C}_\mathfrak{p}(Q)=1. \]

When $p=1$ and $\mathfrak{p}=2$, Problem 4 is the Minkowski problem for  classical Newtonian capacity, which was solved by Jerison
\cite{Minkproblem(Jerison)}, and Caffarelli, Jerison and Lieb
\cite{BMineqCaffarelliJerison}. When $p=1$ and $1<\mathfrak{p}<2$, Problem 4 was solved by CNSXYZ \cite{capacitaryBMtheory(CNSXYZ)}.
For $p>1$, Problem 4 is totally new.

\vskip 5pt

In Section 8, we will solve the general $L_p$ ($p>1$)  Minkowski
problem for $\mathfrak{p}$-capacity (i.e., Problem 5) with $1<\mathfrak{p}\leq 2$, under the basis of Theorem \ref{discreteCapMinkowskiThm}.
To achieve  this goal, our
strategy is first to demonstrate  the duality of Problem 3 and Problem 4, in the sense that their solutions only
differ by a scale factor. Then we  show that Problem 5  is  equivalent to Problem 3,
in the sense that their solutions are identical.

\begin{lemma}\label{lem6.1}
(1) If convex body $K$  solves Problem 3, then convex body
\[\bar{K}=\frac{K}{{\rm C}_\mathfrak{p}(K)^{\frac{1}{n-\mathfrak{p}}}}\]
solves Problem 4.

(2) If  convex body ${\bar K}$  solves Problem 4, then convex body
\[K=\frac{{\bar K}}{F_p({\bar K})^\frac{1}{p}}\]
solves Problem 3.
\end{lemma}

\begin{proof}
(1) Assume that $K$ solves Problem 3. Let $Q$ be a convex body
containing the origin such that ${\rm C}_\mathfrak{p}(Q)=1$.
Since $F_p(K)=1$ and $F_p(\frac{Q}{F_p(Q)^{\frac{1}{p}}})=1$, we have
\begin{align*}
{F_p}\left( {\bar K} \right) = {F_p}\left( {\frac{K}{{{{\rm C}_\mathfrak{p}}{{(K)}^{\frac{1}{{n - \mathfrak{p}}}}}}}} \right) = \frac{{{F_p}(K)}}{{{{\rm C}_\mathfrak{p}}{{(K)}^{\frac{p}{{n - \mathfrak{p}}}}}}} &= \frac{1}{{{{\rm C}_\mathfrak{p}}{{(K)}^{\frac{p}{{n - \mathfrak{p}}}}}}} \\
&\le \frac{1}{{{{\rm C}_\mathfrak{p}}{{\left( {\frac{Q}{{{F_p}{{(Q)}^{\frac{1}{p}}}}}} \right)}^{\frac{p}{{n - \mathfrak{p}}}}}}} = \frac{{{F_p}(Q)}}{{{{\rm C}_\mathfrak{p}}{{\left( Q \right)}^{\frac{p}{{n - \mathfrak{p}}}}}}} = {F_p}(Q),
\end{align*}
which shows that ${\bar K}$ solves Problem 4.

(2) Assume that ${\bar K}$ solves Problem 4. Let $Q$ be a convex
body containing the origin such that $F_p(Q)=1$. Since
${\rm C}_\mathfrak{p}({\bar K})=1$ and ${\rm
C}_\mathfrak{p}(\frac{Q}{{\rm
C}_\mathfrak{p}(Q)^{\frac{1}{n-\mathfrak{p}}}})=1$, we have
\begin{align*}
{{\rm C}_\mathfrak{p}}{(K)^{\frac{p}{{n - \mathfrak{p}}}}} = \frac{{{{\rm C}_\mathfrak{p}}{{(\bar K)}^{\frac{p}{{n - \mathfrak{p}}}}}}}{{{F_p}(\bar K)}} &= \frac{1}{{{F_p}(\bar K)}} \\
&\ge \frac{1}{{{F_p}\left( {\frac{Q}{{{{\rm C}_\mathfrak{p}}{{(Q)}^{\frac{1}{{n - \mathfrak{p}}}}}}}} \right)}} = \frac{{{{\rm C}_\mathfrak{p}}{{(Q)}^{\frac{p}{{n - \mathfrak{p}}}}}}}{{{F_p}\left( Q \right)}} = {{\rm C}_\mathfrak{p}}{(Q)^{\frac{p}{{n - \mathfrak{p}}}}},
\end{align*}
which shows that $K$ solves Problem 3.
\end{proof}

\begin{lemma}\label{lem6.2}
If $\mu$ is a discrete measure, then Problem 3 and Problem 2 are identical.
\end{lemma}

\begin{proof}
Assume that $\mu$ is a discrete measure, say, $\mu=\sum_i^m
c_i\delta_{\xi_i}$. For any convex body $Q$ containing the origin,
since $ {\scriptstyle \pmb \lbrack} h_Q|_{{\rm supp}\;\mu}
{\scriptstyle \pmb \rbrack} \supseteq Q$, it follows that ${\rm
C}_\mathfrak{p}({\scriptstyle \pmb \lbrack} h_Q|_{{\rm supp}\;\mu}
{\scriptstyle \pmb \rbrack} )\ge {\rm C}_\mathfrak{p}(Q)$. Since
$F({\scriptstyle \pmb \lbrack} h_Q|_{{\rm supp}\;\mu} {\scriptstyle
\pmb \rbrack} )=F(Q)=1$, it follows that the domain of Problem 3 can
be restricted to the class of proper convex polytopes $P(y)$
generated by
\[ P(y)=\bigcap_{i=1}^m\left\{x\in\mathbb{R}^n, x\cdot\xi_i\le y_i \right\}, \]
for $y=(y_1,\cdots,y_m)\in\mathbb{R}^m_+$.
\end{proof}

Therefore, for a discrete measure $\mu$, Problem 3 and Problem 2, even further as well as Problem 1 (i.e., the
discrete $L_p$ Minkowski problem for $\mathfrak{p}$-capacity) have the
same unique solution. A generalization of Problem 1 is as
follows.

\vskip 10pt

\noindent \textbf{Problem 5.} Among all convex bodies in
$\mathbb{R}^n$ that contain the origin, find a body $K$ such that
\[\frac{d\mu_{\mathfrak{p}}(K,\cdot)}{{\rm C}_{\mathfrak{p}}(K)}=h_K^{p-1}d\mu. \]

The equivalence between Problem 3 and Problem 5 is shown by the next lemma.

\begin{lemma}\label{lemma6.3}
Let  $1< p<\infty$  and $1<\mathfrak{p}<n$. Suppose that  $\mu$ is a
finite Borel measure on $\mathbb{S}^{n-1}$ and is not concentrated
on any closed hemisphere. Then a convex body $K$   solves Problem 3,
if and only if $K$  solves Problem 5. Moreover, if Problem 5 (or
equivalently, Problem 3) has a solution, then such solution is
unique.
\end{lemma}

\begin{proof}
First, assume that $K$ solves Problem 3. We  prove that $K$ also
solves Problem 5.

Let $f\in C(\mathbb{S}^{n-1})$ be nonnegative.  For $t\ge 0$, let
\[K_t={\scriptstyle \pmb
\lbrack} h_K+tf {\scriptstyle \pmb \rbrack} \quad {\rm and }\quad
F_p(h_K+tf)=\frac{{\mathfrak p}-1}{n- {\mathfrak
p}}\int\limits_{\mathbb{S}^{n-1}}(h_K+tf)^pd\mu.\] Then,
$F_p(h_K+tf)\ge F_p(K_t)$. Since $K$ solves Problem 3, and
$F_p(K_t)^{-\frac{1}{p}}K_t$ satisfies the constraint in Problem 3,
it follows that for $t\ge 0$,
\[G(t):={\rm C}_{\mathfrak p}\left( \frac{K_t}{F_p(h_K+t f)^{\frac{1}{p}}} \right)\le {\rm C}_{\mathfrak p}(K).\]
Clearly,  $G(t)$ is continuous in $t\ge 0$, and $G(0)={\rm
C}_{\mathfrak p}(K)$. Since
 \[{\left. {\frac{{d{F_p}({h_K} + tf)}}{{dt}}} \right|_{t = {0^ + }}} = \frac{{p(\mathfrak{p} - 1)}}{{n - \mathfrak{p}}}\int\limits_{{\mathbb{S}^{n - 1}}} {fh_K^{p - 1}d\mu } \]
and
\[{\left. {\frac{{d{{\rm{C}}_{\mathfrak{p}}}({K_t})}}{{dt}}} \right|_{t = {0^ + }}} = (\mathfrak{p} - 1)\int\limits_{{\mathbb{S}^{n - 1}}} {fd{\mu _{\mathfrak{p}}}(K, \cdot )} ,\]
it follows  that
\[0=G'_+(0)=(\mathfrak{p} - 1)\int\limits_{{\mathbb{S}^{n - 1}}} {fd{\mu _\mathfrak{p}}(K, \cdot )}  - (\mathfrak{p} - 1){{\rm{C}}_\mathfrak{p}}(K)\int\limits_{{\mathbb{S}^{n - 1}}} {fh_K^{p - 1}d\mu } .\]
Thus,
\[\int\limits_{{\mathbb{S}^{n - 1}}} {fh_K^{p - 1}d\mu }  = \frac{1}{{{{\rm{C}}_\mathfrak{p}}(K)}}\int\limits_{{\mathbb{S}^{n - 1}}} {fd{\mu _\mathfrak{p}}(K, \cdot )} .\]
That is, the above equality  holds for any nonnegative $f\in
C(\mathbb{S}^{n-1})$. Therefore, it also holds for any  $f \in
C(\mathbb{S}^{n-1})$, which concludes that ${\rm C}_{\mathfrak
p}(K)^{-1}d\mu_{\mathfrak{p}}(K,\cdot)=h_K^{p-1}d\mu$.

Conversely, assume that  $K$ solves Problem 5.  Let $Q$ be a convex
body containing the origin, such that $1=\frac{{\mathfrak
p}-1}{n-{\mathfrak p}}\int_{\mathbb{S}^{n-1}}h_Q^pd\mu$. Our aim is
to  prove that ${\rm C}_\mathfrak{p}(K)\ge {\rm C}_\mathfrak{p}(Q)$.
That is,   $K$ also solves Problem 3.

Using the condition that ${\rm C}_\mathfrak{p}(K)h_K^{p-1}d\mu
=d\mu_\mathfrak{p}(K,\cdot)$, we have
\begin{align*}
1 &= \frac{{\mathfrak{p} - 1}}{{n - \mathfrak{p}}}\int\limits_{\{ {h_K} > 0\} } {h_Q^pd\mu }  + \frac{{\mathfrak{p} - 1}}{{n - \mathfrak{p}}}\int\limits_{\{ {h_K} = 0\} } {h_Q^pd\mu }  \\
&\ge \frac{{\mathfrak{p} - 1}}{{n - \mathfrak{p}}}\int\limits_{\{ {h_K} > 0\} } {h_Q^pd\mu }  \\
&= \frac{{\mathfrak{p} - 1}}{{n - \mathfrak{p}}}\int\limits_{\{ {h_K} > 0\} } {{{\left( {\frac{{{h_Q}}}{{{h_K}}}} \right)}^p}\frac{{{h_K}}}{{{{\rm C}_\mathfrak{p}}(K)}}d{\mu _\mathfrak{p}}(K, \cdot )}.
\end{align*}

From the Poincar$\rm{\acute{e}}$ $\mathfrak p$-capacity formula, it follows that
\[{{\rm C}_\mathfrak{p}}(K) = \frac{{\mathfrak{p} - 1}}{{n - \mathfrak{p}}}\int\limits_{\{ {h_K} > 0\} } {{h_K}d{\mu _\mathfrak{p}}(K, \cdot )}.\]
So, the measure $\frac{\mathfrak{p}-1}{(n-\mathfrak{p}){\rm
C}_\mathfrak{p}(K)}h_Kd\mu_\mathfrak{p}(K,\cdot)$ is a Borel
probability measure on the set $\{h_K\neq 0\}$. From the  Jensen
inequality,  we have
\begin{align*}
1 &\ge {\left( {\frac{{\mathfrak{p} - 1}}{{(n - \mathfrak{p}){{\rm C}_\mathfrak{p}}(K)}}\int\limits_{\{ {h_K} > 0\} } {{{\left( {\frac{{{h_Q}}}{{{h_K}}}} \right)}^p}{h_K}d{\mu _\mathfrak{p}}(K, \cdot )} } \right)^{\frac{1}{p}}} \\
&\ge \frac{{\mathfrak{p} - 1}}{{(n - \mathfrak{p}){{\rm C}_\mathfrak{p}}(K)}}\int\limits_{\{ {h_K} > 0\} } {\frac{{{h_Q}}}{{{h_K}}}{h_K}d{\mu _\mathfrak{p}}(K, \cdot )}  \\
&= \frac{{\mathfrak{p} - 1}}{{(n - \mathfrak{p}){{\rm C}_\mathfrak{p}}(K)}}\int\limits_{\{ {h_K} > 0\} } {{h_Q}d{\mu _\mathfrak{p}}(K, \cdot )}.
\end{align*}
Furthermore, from the $\mathfrak{p}$-capacitary Minkowski
inequality, we have
\begin{align*}
1 &\ge  \frac{{{{\rm{C}}_\mathfrak{p}}(K,Q)}}{{{{\rm{C}}_\mathfrak{p}}(K)}} - \frac{{\mathfrak{p} - 1}}{{(n - \mathfrak{p}){{\rm C}_\mathfrak{p}}(K)}}\int\limits_{\{ {h_K} = 0\} } {{h_Q}d{\mu _\mathfrak{p}}(K, \cdot )}  \\
&\ge {\left( {\frac{{{{\rm{C}}_\mathfrak{p}}(Q)}}{{{{\rm{C}}_\mathfrak{p}}(K)}}} \right)^{\frac{1}{{n - \mathfrak{p}}}}} - \frac{{\mathfrak{p} - 1}}{{(n - \mathfrak{p}){{\rm C}_\mathfrak{p}}(K)}}\int\limits_{\{ {h_K} = 0\} } {{h_Q}d{\mu _\mathfrak{p}}(K, \cdot )} .
\end{align*}
By the condition that ${\rm C}_\mathfrak{p}(K)h_K^{p-1}d\mu =d\mu_\mathfrak{p}(K,\cdot)$, it follows   that
 \[\int\limits_{\{ {h_K} = 0\} } {{h_Q}d{\mu _\mathfrak{p}}(K, \cdot )}  = \int\limits_{\{ {h_K} = 0\} } {{h_Q}h_K^{p - 1}d\mu }  = 0.\]
Thus,
\[1 \ge {\left( {\frac{{{{\rm{C}}_\mathfrak{p}}(Q)}}{{{{\rm{C}}_\mathfrak{p}}(K)}}} \right)^{\frac{1}{{n - \mathfrak{p}}}}},\]
as desired.

It remains to prove that if $K$ and $L$ are  solutions to Problem 5,
then $K=L$. From the above argument and the equality condition of
the $\mathfrak{p}$-capacitary Minkowski inequality, we see that $K$
and $L$ are homothetic, so that ${\rm C}_{\mathfrak p}(K)={\rm
C}_{\mathfrak p}(L)$. In other words, $K=L+x$, for some
$x\in\mathbb{R}^n$. From the translation invariance of
$\mathfrak{p}$-capacitary measure and the assumptions, it follows
that $${\left( {{h_L}(\xi ) + x \cdot \xi } \right)^{p - 1}}d\mu
(\xi ) = {h_L}{(\xi )^{p - 1}}d\mu (\xi ).$$ In other words,
\begin{equation}\label{contradiction2}
{\left( {{h_L}(\xi ) + x \cdot \xi } \right)^{p - 1}} = {h_L}{(\xi
)^{p - 1}},\quad {\rm for}\; \mu \rm{-almost\; all}\;
\xi\in\mathbb{S}^{n-1}.
\end{equation}
Note that $\mu$ is not concentrated on any closed hemisphere. If $x$
is nonzero, then on the open hemisphere
$U:=\{\xi\in\mathbb{S}^{n-1}: x\cdot\xi>0\}$, we have $\mu(U)>0$ and
${\left( {{h_L}(\xi ) + x \cdot \xi } \right)^{p - 1}} > {h_L}{(\xi
)^{p - 1}}$, for all $\xi\in U$, which contradicts
(\ref{contradiction2}). Hence, $K=L$.

The proof is complete.
\end{proof}

By now,  we propose 5 related problems in variant disguises. For convenience, it is necessary to summarize their relationship here.

(1).  Problem 1 and Problem 2 are proposed  exclusively for \emph {discrete measures}. They have the identical unique solution;

(2).  Problem 3 and Problem 4 are \emph{dual} each other. Their solutions only  differ by a scale factor.
For discrete measures, Problem 3 and Problem 2 are identical.

(3).  Problem 5  generalizes  Problem 1 to \emph{general measures}.

(4).  Problem 5 and Problem 3  are equivalent. They have the identical unique solution.

\vskip 25pt
\section{\bf Several useful lemmas for Section 8}
\vskip 10pt

In light of the equivalence of Problem 3 and Problem 5, we will solve Problem 5 in Section 8 via the passage by  firstly  solving Problem 3.
For this aim, we have to make more preparatory works.
Throughout this section, let
$1<p<\infty$ and $1<\mathfrak{p}<n$.

Suppose that $\mu$ and $\mu_j$,
$j\in\mathbb{N}$, are finite Borel measures on
$\mathbb{S}^{n-1}$ and  not concentrated on any closed
hemisphere.  For each $j$,  assume that $K_j$ is the solution to Problem 5 for $\mu_j$.

Let
\[\bar{K}_j=\frac{K_j}{{\rm C}_\mathfrak{p}(K_j)^{\frac{1}{n-\frak{p}}}}.\]
From Lemma \ref{lemma6.3} and Lemma \ref{lem6.1} (1), it  implies that
$\bar{K_j}$ is the solution to Problem 4 for $\mu_j$.

For a  convex body $Q$ in  $\mathbb{R}^n$ containing the origin, let
\[F_{p,j}(Q)=\frac{\mathfrak{p}-1}{n-\mathfrak{p}}\int\limits_{\mathbb{S}^{n-1}}h_Q^pd\mu_j\quad {\rm and}\quad F_p(Q)=
\frac{\mathfrak{p}-1}{n-\mathfrak{p}}\int\limits_{\mathbb{S}^{n-1}}h_Q^pd\mu. \]

\begin{lemma}\label{lemma7.1}
If $\{\mu_j\}_j$ converges weakly to $\mu$, then $\{K_j\}_j$ and
$\{\bar{K}_j\}_j$ are  bounded from above.
\end{lemma}

\begin{proof}
For each $j$,
there is a $\xi_j\in\mathbb{S}^{n-1}$ such that
$h_{K_j}(\xi_j)=\max_{\mathbb{S}^{n-1}}h_{K_j}$. Since the segment
joining the origin and
$(\max_{\mathbb{S}^{n-1}}h_{K_j})\xi_j$ is contained in $K_j$, it
follows that for all $\xi\in\mathbb{S}^{n-1}$,
\[(\max_{\mathbb{S}^{n-1}}h_{K_j})(\xi_j\cdot\xi)_+\le h_{K_j}(\xi),\] where $(\xi_j\cdot\xi)_+=\max\{0, \xi_j\cdot\xi\}$.
Thus,
\begin{align*}
1 &= \frac{\mathfrak{p}-1}{n-\mathfrak{p}}\int\limits_{{\mathbb{S}^{n - 1}}} {h_{{K_j}}^pd{\mu _j}}  \\
&\ge {({\max _{{\mathbb{S}^{n - 1}}}}{h_{{K_j}}})^p}\frac{\mathfrak{p}-1}{n-\mathfrak{p}}\int\limits_{{\mathbb{S}^{n - 1}}} {{{({\xi _j} \cdot \xi )}_+^p }d{\mu _j}(\xi )}  \\
&\ge {({\max _{{\mathbb{S}^{n - 1}}}}{h_{{K_j}}})^p}\frac{\mathfrak{p}-1}{n-\mathfrak{p}}{\min _{\xi ' \in {\mathbb{S}^{n - 1}}}}\int\limits_{{\mathbb{S}^{n - 1}}} {{{(\xi ' \cdot \xi )}_+^p }d{\mu _j}(\xi )} .
\end{align*}

Consider the functional $\mathbb{R}^{n}\to \mathbb{R}$,
\[x \mapsto {\left( {\frac{{\mathfrak{p} - 1}}{{n - \mathfrak{p}}}\int\limits_{{\mathbb{S}^{n - 1}}} {(x \cdot \xi )_ + ^pd{\mu _j}(\xi )} } \right)^{\frac{1}{p}}}.\]
Since $((x+x')\cdot\xi)_+\le (x\cdot\xi)_++(x'\cdot\xi)_+$, by the Minkowski integral inequality, it
implies that this functional is convex. Since $\mu_j$ is not concentrated on
any closed hemisphere, this functional is strictly positive for any
nonzero $x$. Thus, this functional is the support function of a
unique convex body, say $\Pi_{\mathfrak{p},p}\mu_j\in\mathcal{K}^n_o$.
So,
$\min_{\mathbb{S}^{n-1}}h_{\Pi_{\mathfrak{p},p}\mu_j}>0$
and
\[{\max _{{\mathbb{S}^{n - 1}}}}{h_{{K_j}}} \le \frac{1}{{{{\min }_{{\mathbb{S}^{n - 1}}}}{h_{{\Pi_{\mathfrak{p},p}\mu_j}}}}} < \infty .\]

Similarly,  define the convex body
$\Pi_{\mathfrak{p},p}\mu\in\mathcal{K}^n_o$ by
\[{h_{{\Pi _{\mathfrak{p},p}}\mu }}(x) = {\left( {\frac{{\mathfrak{p} - 1}}{{n - \mathfrak{p}}}\int\limits_{{\mathbb{S}^{n - 1}}} {(x \cdot \xi )_ + ^pd{\mu }(\xi )} } \right)^{\frac{1}{p}}}.\]

Since the weak convergence $\mu_j\to\mu$ yields the pointwise convergence
$h_{\Pi_{\mathfrak{p},p}\mu_j}\to h_{\Pi_{\mathfrak{p},p}\mu}$ on
$\mathbb{S}^{n-1}$, and  the pointwise convergence of support
functions on $\mathbb{S}^{n-1}$ is also a uniform convergence,  it follows that
the sequence $\{h_{\Pi_{\mathfrak{p},p}\mu_j}\}_j$ on
$\mathbb{S}^{n-1}$ is uniformly bounded from below by a constant
$m>0$. So, we have
\[\sup_j {\left\{ {{{\max }_{{\mathbb{S}^{n - 1}}}}{h_{{K_j}}}} \right\}} \le \frac{1}{{\inf_j {{\left\{ {{{\min }_{{\mathbb{S}^{n - 1}}}}{h_{{\Pi _{\mathfrak{p},p}}{\mu _j}}}} \right\}}}}} \le \frac{1}{m} < \infty ,\]
which implies that $\{K_j\}_j$ is bounded from above.

To prove that $\{\bar{K}_j\}_j=\{\frac{K_j}{{\rm
C}_\mathfrak{p}(K_j)^{\frac{1}{n-\mathfrak{p}}}}\}_j$ is also bounded from above, two
observations are in order. First, by the
fact that $F_{p,j}\left(( \frac{\mathfrak{p} - 1}{n - \mathfrak{p}}|\mu _j| )^{{-1}/{p}} B \right) = 1 $, where $|\mu_j|$
denotes the total mass of $\mu_j$, the ball $( \frac{\mathfrak{p} -
1}{n - \mathfrak{p}}|\mu _j| )^{{-1}/{p}} B$ satisfies the
constraint in Problem 3 for $\mu_j$. Thus,
\[{\rm C}_\mathfrak{p}(K_j)\ge {\rm C}_\mathfrak{p}\left((
\frac{\mathfrak{p} - 1}{n - \mathfrak{p}}|\mu _j| )^{\frac{-1}{p}} B\right).\]
Second, the weak convergence $\mu_j\to\mu$ yields the convergence
$|\mu_j|\to |\mu|$, which implies that
\[\sup_j\{|\mu_j|\}<\infty.\]

So,
\begin{align*}
{\max _{{\mathbb{S}^{n - 1}}}}{h_{{{\bar K}_j}}} &= \frac{{{{\max }_{{\mathbb{S}^{n - 1}}}}{h_{{K_j}}}}}{{{\rm
C}_\mathfrak{p}{{({K_j})}^{\frac{1}{{n- \mathfrak{p}}}}}}}\\
& \le \frac{{{{\max }_{{\mathbb{S}^{n - 1}}}}{h_{{K_j}}}}}{{{\rm C}_\mathfrak{p}{{\left( \left(
\frac{\mathfrak{p} - 1}{n - \mathfrak{p}}|\mu _j| \right)^{\frac{-1}{p}} B \right)}^{\frac{1}{{n - \mathfrak{p}}}}}}}\\
& = \frac{(\mathfrak{p}-1)^{\frac{1}{p}}{|{\mu_j}|^{\frac{1}{p}}{{\max }_{{\mathbb{S}^{n - 1}}}}{h_{{K_j}}}}}{(n-\mathfrak{p})^{\frac{1}{p}}{{\rm
C}_\mathfrak{p}{{\left( B \right)}^{\frac{1}{{n -\mathfrak{p}}}}}}}\\
&\le
M:=\left(\frac{\mathfrak{p}-1}{n-\mathfrak{p}}\right)^{\frac{1}{p}}
{\rm C}_\mathfrak{p}(B)^{\frac{-1}{n-\mathfrak{p}}}
\sup_j\{|\mu_j|\}^{\frac{1}{p}} \sup_j\{\max_{\mathbb{S}^{n-1}}h_{K_j}\}\\
&<\infty,
\end{align*}
which concludes that $\{\bar{K}_j\}_j$ is bounded from above.
\end{proof}

\begin{lemma}\label{lemma7.2}
If $\{\bar{K}_j\}_j$ converges to a compact convex set $\bar{K}$, then $\dim(\bar{K})\neq n-1$.
\end{lemma}

\begin{proof}
Recall that $\bar{K_j}$ is the solution to Problem 4 for $\mu_j$. By Lemma  \ref{lemma6.3} and Lemma \ref{lem6.1} (2), $K_j
=F_{p,j}(\bar{K}_j)^{-1/p}\bar{K}_j$ is the solution to Problem 5  for $\mu_j$. Since ${\rm
C}_\mathfrak{p}(K_j)h_{K_j}^{p-1}d\mu_{j}=d\mu_\mathfrak{p}(K_j,\cdot)$,
it follows that
\[{\rm C}_\mathfrak{p}\left( {\frac{{{{\bar K}_{j}}}}{{{F_{p,j}}{{\left( {{{\bar K}_{j}}} \right)}^{\frac{1}{p}}}}}} \right)h_{\frac{{{{\bar K}_{j}}}}{{{F_{p,j}}{{\left( {{{\bar K}_{j}}} \right)}^{\frac{1}{p}}}}}}^{p - 1}d{\mu _{j}} = d{\mu _\mathfrak{p}}\left( {\frac{{{{\bar K}_{j}}}}{{{F_{p,j}}{{\left( {{{\bar K}_{j}}} \right)}^{\frac{1}{p}}}}}, \cdot } \right).\]
From this, the fact that ${\rm C}_\mathfrak{p}({\bar K}_{j})=1$,
together with the positive homogeneity of $\mathfrak{p}$-capacity,
support functions and $\mathfrak{p}$-capacitary measure, it follows
that
\[h_{{{\bar K}_{j}}}^{p - 1}d{\mu _{j}} = {F_{p,j}}{\left( {{{\bar K}_{j}}} \right)}d{\mu _\mathfrak{p}}\left( {{{\bar K}_{j}}, \cdot } \right).\]
By CNSXYZ \cite[Lemma 2.18]{capacitaryBMtheory(CNSXYZ)}, there is a
positive constant $c$ depending on $n$, $\mathfrak{p}$ and $M$, such
that $ \mu_\mathfrak{p}(\bar{K}_{j},\cdot)\ge
c^{-\mathfrak{p}}S_{\bar{K}_{j}}$. Thus,
\[h_{{{\bar K}_{j}}}^{p - 1}d{\mu _{j}} \ge {c^{ - \mathfrak{p}}}F_{p,j}({{\bar K}_{j}})d{S_{{{\bar K}_{j}}}}.\]

Let $f\in C(\mathbb{S}^{n-1})$ be non-negative. Then,
\begin{equation}\label{ineq7.2.4.1}
\int\limits_{{\mathbb{S}^{n - 1}}} {fh_{{{\bar K}_{j}}}^{p - 1}d{\mu_{j}}}   \ge {c^{ - \mathfrak{p}}}F_{p,j}({{\bar K}_{j}})\int\limits_{{\mathbb{S}^{n - 1}}} {fd{S_{{{\bar K}_{j}}}}}.
\end{equation}
Here, several facts are in order. First, the convergence
$\bar{K}_{j}\to \bar{K}$ is equivalent to the uniform convergence
$h_{\bar{K}_{j}}\to h_{\bar{K}}$ over the sphere $\mathbb{S}^{n-1}$.
Second, the uniform convergence $h_{\bar{K}_{j}}\to h_{\bar{K}}$
together with the weak convergence $\mu_{j}\to \mu$ yields the
convergence $F_{p,j}(\bar{K}_{j})\to F_p(\bar{K})$. Third, the
convergence $\bar{K}_{j}\to \bar{K}$ again yields the weak
convergence $S_{\bar{K}_{j}}\to S_{\bar{K}}$. Hence, let
$j\to\infty$,   (\ref{ineq7.2.4.1}) yields that
\begin{equation}\label{ineq7.2.4.2}
\int\limits_{{\mathbb{S}^{n - 1}}} {fh_{\bar{K}}^{p - 1}d\mu }  \ge
{c^{ - \mathfrak{p}}}F_p(\bar K)\int\limits_{{\mathbb{S}^{n - 1}}}
{fd{S_{\bar K}}} .
\end{equation}

With this inequality in hand, we devote to showing that $\dim(K)\neq n-1$.

Assume that $\dim(\bar{K})= n-1$ and $\bar K$ is contained in an
$(n-1)$-dimensional linear subspace with normal
$\xi_0\in\mathbb{S}^{n-1}$. By the definition of surface area measure,
$S_{\bar K}=V_{n-1}(\bar K)(\delta_{\xi_0}+\delta_{-\xi_0})$, where
 $V_{n-1}(\bar{K})$ is the $(n-1)$-dimensional volume of
$\bar{K}$. Now, (\ref{ineq7.2.4.2}) can be reformulated as
\begin{equation}\label{ineq7.2.4.3}
\int\limits_{{\mathbb{S}^{n - 1}}} {f d\bar{\mu} }  \ge c' \cdot (f({\xi _0}) + f( - {\xi _0})),
\end{equation}
where $\bar{\mu}$ is the Borel measure on $\mathbb{S}^{n-1}$ defined
by $d\bar{\mu}=h_{\bar{K}}^{p-1}d\mu$, and $ {c'} = {c^{ - \mathfrak{p}}}F(\bar K){V_{n - 1}}(\bar K)$.

Recall that $\bar{K}$ contains the origin. So, $h_{\bar{K}}\ge 0$,
which in turn gives  $F_p(\bar{K})\ge 0$.  Now,  we  prove that
$F_p(\bar{K})>0.$
Assume that $F_p(\bar{K})= 0$. Since
\[0=\int\limits_{\mathbb{S}^{n-1}}h_{\bar{K}}^pd\mu=\int\limits_{\{h_{\bar{K}}>0\}}h_{\bar{K}}^pd\mu+\int\limits_{\{h_{\bar{K}}=0\}}h_{\bar{K}}^pd\mu=\int\limits_{\{h_{\bar{K}}>0\}}h_{\bar{K}}^pd\mu,\]
it follows that $\mu(\{h_{\bar{K}}>0\})=0$. Thus,
\[{\rm supp} \mu\subseteq \mathbb{S}^{n-1}\setminus\{h_{\bar{K}}>0\}=\{h_{\bar{K}}=0\}. \]
Since $\{h_{\bar{K}}=0\}$ is contained in some closed
hemisphere, it follows that $\mu$ is concentrated on some closed
hemisphere, which is a contradiction. Hence, $F_p(\bar{K})>0$,
and therefore,  $c'>0$.

With $c'>0$ and $f\in C(\mathbb{S}^{n-1})$ is  non-negative, by Evans
and Gariepy \cite[Theorem 3, p. 42]{BookEvansGariepy},  (\ref{ineq7.2.4.3}) implies that the Borel measure
$\bar\mu$ satisfies
\[\bar{\mu}(\{\xi_0\})=\bar{\mu}(\{-\xi_0\})>0. \]
However, from the assumption that $h_{\bar{K}}(\pm\xi_0)=0$ and the
definition of $\bar{\mu}$, it follows that
\[\bar{\mu}(\{\xi_0\})=\bar{\mu}(\{-\xi_0\}) =0. \]
A contradiction occurs. Hence, $\dim(K)\neq n-1$.
\end{proof}

\begin{lemma}\label{lemma7.3}
Suppose   $1<\mathfrak{p}\le 2$.  If
$\{\bar{K}_j\}_j$ converges to a compact convex set $\bar{K}$, then
$\dim(\bar{K})\neq 0,1,\ldots, n-2$.
\end{lemma}

\begin{proof}
The arguments here is similar to that from CNSXYZ \cite[p.
1571]{capacitaryBMtheory(CNSXYZ)}.  If $1<\mathfrak{p}\le 2$ and
$\dim(\bar K)\le n-2$, then $\dim(\bar K)\le n-\mathfrak{p}$ and
thus $\mathcal{H}^{n-\mathfrak{p}}(\bar K)<\infty$. According to  Evans and Gariepy
\cite[Theorem 3, p. 154]{BookEvansGariepy}: if
$\mathcal{H}^{n-\mathfrak{p}}( \bar K)<\infty$, then ${\rm
C}_\mathfrak{p}( \bar K)=0$, it follows that ${\rm C}_\mathfrak{p}(\bar K)=0$.
This is impossible, because of the continuity of ${\rm
C}_\mathfrak{p}$ and the fact that ${\rm
C}_\mathfrak{p}(\bar{K}_{j})=1$ for each $j$.
\end{proof}

\begin{lemma}\label{lemma7.4}
Suppose $1<\mathfrak{p}\le 2$. If
$\{\bar{K}_j\}_j$ converges to a compact convex set $\bar{K}$, then
the following assertions hold.

\noindent (1) $\bar{K}$ is a convex body containing the origin.

\noindent (2) $0<\int_{\mathbb{S}^{n-1}}h_{\bar{K}}^pd\mu<\infty$.

\noindent (3) The convex body
\[K=\left(\frac{\mathfrak{p}-1}{n-\mathfrak{p}}\int\limits_{\mathbb{S}^{n-1}}h_{\bar{K}}^pd\mu\right)^{\frac{-1}{p}}\bar{K}\]
is the unique solution to Problem 5 for $\mu$.
\end{lemma}

\begin{proof}
By Lemma \ref{lemma7.2} and Lemma \ref{lemma7.3}, it follows that $\bar{K}$ is a convex body containing the origin.

From the facts that $\max_{\mathbb{S}^{n-1}}h_{\bar{K}}^p<\infty$ and
 $|\mu|<\infty$, it   follows that $ \int_{\mathbb{S}^{n-1}}h_{\bar{K}}^pd\mu<\infty$.
Now,  we show
$\int_{\mathbb{S}^{n-1}}h_{\bar{K}}^pd\mu>0$ by contradiction. Assume
that $\int_{\mathbb{S}^{n-1}}h_{\bar{K}}^pd\mu=0$. Then, $
0=\int_{\{h_{\bar{K}}>0\}}h_{\bar{K}}^pd\mu$, and therefore,
$\mu(\{h_{\bar{K}}>0\})=0$. If  $\bar{K}$ contains the
origin in its interior, then $\{h_{\bar{K}}>0\}=\mathbb{S}^{n-1}$
and  $\mu(\{h_{\bar{K}}>0\})=\mu(\mathbb{S}^{n-1})=|\mu|>0$. So, the
origin is on the  boundary  of $\bar{K}$, and therefore
$\{h_{\bar{K}}=0\}$ is contained in some closed hemisphere. Note
that $ {\rm supp} \mu\subseteq \{h_{\bar{K}}=0\}$. So, $\mu$ is
concentrated on some closed hemisphere. It is a contradiction.

The assertions (1) and (2) imply that $K$ is a convex body
containing the origin. Since
\[K_j=\left(\frac{p-1}{n-p}\int\limits_{\mathbb{S}^{n-1}}h_{\bar{K}_j}^pd\mu_j\right)^{-1/p}
\bar{K}_j \quad {\rm and}\quad \lim_{j\to\infty}
\int\limits_{\mathbb{S}^{n-1}}h_{\bar{K}_j}^pd\mu_j=\int\limits_{\mathbb{S}^{n-1}}h_{\bar{K}}^pd\mu_j,\]
it follows that $\{K_j\}_j$ converges to $K$. From ${\rm
C}_\mathfrak{p}(K_j)h_{K_j}^{p-1}d\mu_{j}=d\mu_\mathfrak{p}(K_j,\cdot)$,
and  the facts that the uniform convergence $h_{K_j}\to h_K$ yields
the convergence ${\rm C}_\mathfrak{p}(K_j)\to {\rm
C}_\mathfrak{p}(K)$ and the weak convergence
$\mu_\mathfrak{p}(K_j,\cdot)\to \mu_\mathfrak{p}(K,\cdot)$, it
follows that $ {\rm
C}_\mathfrak{p}(K)h_K^{p-1}d\mu=d\mu_\mathfrak{p}(K,\cdot)$. So, $K$
is a solution to Problem 5 for $\mu$. As far the uniqueness, it is
guaranteed  by Lemma \ref{lemma6.3}.
\end{proof}

\vskip 25pt
\section{\bf  The $\boldsymbol{L_p}$ Minkowski problem for $\boldsymbol{\mathfrak{p}}$-capacity when $\boldsymbol{1<\mathfrak{p}\le 2}$}
\vskip 10pt

With the  preparatory works in Section 6 and Section 7, we set out to prove Theorem 1.2.

\begin{theorem}\label{generalCapMProblem}
Suppose  $1< p<\infty$ and $1<\mathfrak{p}\le 2$.
If $\mu$ is a finite Borel measure on $\mathbb{S}^{n-1}$
which is not concentrated on any closed hemisphere, then there
exists a unique convex body $K$ in $\mathbb{R}^n$ containing the
origin, such that $ {\rm
C}_\mathfrak{p}(K)h_K^{p-1}d\mu=d\mu_\mathfrak{p}(K,\cdot)$.
If in addition $p\ge n$, then $K$ contains the origin in its interior.
\end{theorem}

\begin{proof}
Take a sequence of discrete measures $\{\mu_j\}_j$ on
$\mathbb{S}^{n-1}$, such that each $\mu_j$ is not concentrated on
any closed hemisphere and $\mu_j\to\mu$ weakly. By Theorem
\ref{discreteCapMinkowskiThm} and Lemma \ref{lem6.2}, for each $j$,
Problem 5 for $\mu_j$ has a unique solution $P_j$, a convex polytope
containing the origin in its interior.

 Let
\[\bar{P_j}=\frac{P_j}{{\rm
C}_\mathfrak{p}(P_j)^{\frac{1}{n-\mathfrak{p}}}}.\]
By Lemma
\ref{lemma6.3} and Lemma \ref{lem6.1}, $\bar{P}_j$ is the unique
solution to Problem 4 for $\mu_j$. Since $\mu_j\to\mu$ weakly, the sequence $\{\bar{P}_j\}_j$ is bounded from
above  by
Lemma \ref{lemma7.1}. From the Blaschke selection theorem, $\{\bar{P}_j\}_j$ has a
convergent subsequence $\{{\bar P}_{j_l}\}_l$, which converges to a
compact convex set, say $\bar{K}$. By Lemma \ref{lemma7.4} (1),
$\bar{K}$ is a convex body containing the origin. By Lemma \ref{lemma7.4} (2),
$0<\int_{\mathbb{S}^{n-1}}h_{\bar{K}}^pd\mu<\infty$. Thus, we get a
convex body
\[K:=\left(\frac{\mathfrak{p}-1}{n-\mathfrak{p}}\int\limits_{\mathbb{S}^{n-1}}h_{\bar{K}}^pd\mu\right)^{\frac{-1}{p}}\bar{K}. \]
By Lemma \ref{lemma7.4} (3), the convex body $K$ is the unique
solution to Problem 5 for $\mu$.

It remains to prove that if in addition $p\ge n$,  then $K$ contains the origin in its interior.

Several useful facts are listed. First,  $\sup_l\{|\mu_{j_l}|\}<\infty$.
Second, $d\mu_{j_l}=\frac{h_{P_{j_l}}^{1-p}}{{\rm
C}_\mathfrak{p}(P_{j_l})}d\mu_{\mathfrak{p}}(P_{j_l},\cdot)$, for
each $l$. Third, from the convergence $P_{j_l}\to K$ and  CNSXYZ
\cite[Lemma 2.18]{capacitaryBMtheory(CNSXYZ)}, there is a positive
constant $c_1$ depending on $n$, $\mathfrak{p}$ and $
\max\{h_{P_{j_l}}(\xi): \xi\in\mathbb{S}^{n-1}, l\in\mathbb{N}\}$,
such that $ \mu_\mathfrak{p}(P_{j_l},\cdot)\ge
c_1^{-\mathfrak{p}}S_{\bar{P}_{j_l}}$. Finally, from the convergence
$P_{j_l}\to K$ again and the continuity of $\mathfrak{p}$-capacity,
it follows that $0< \sup_l\{{\rm C}_\mathfrak{p}(P_{j_l})\}<\infty$.
Hence,
\[\infty  > {\sup _l}\left\{ {|{\mu _{{j_l}}}|} \right\} \ge |{\mu _{{j_l}}}|
= \frac{1}{{{\rm C}_\mathfrak{p}({P_{{j_l}}})}}\int\limits_{{\mathbb{S}^{n - 1}}} {h_{{P_{{j_l}}}}^{1 - p}d{\mu _\mathfrak{p}}({P_{{j_l}}}, \cdot )}
\ge {c_2}\int\limits_{{\mathbb{S}^{n - 1}}} {h_{{P_{{j_l}}}}^{1 - p}d{S_{{P_{{j_l}}}}}} ,\]
where ${c_2}= \frac{{c_1^{ - \mathfrak{p}}}}{{{{\sup }_l}\left\{ {{\rm C}_\mathfrak{p}({P_{{j_l}}})} \right\}}}$.

Assume that the origin is on the boundary  of $K$. We derive that $p<n$ by adapting an argument from Hug and LYZ \cite[p.713]{Minkproblem(HugLYZ)}. Let
$\xi_K\in\mathbb{S}^{n-1}$ be  such that
$\partial K$ can be locally represented as the graph of a convex
function over (a neighborhood of) $B_r:=\xi_K^\bot\cap rB$, $r>0$, and
$x\cdot\xi_K\ge 0$ for any $x\in K$.  There exists a subsequence
$\{j_{l_k}\}_k$ of $\{j_l\}_l$ tending to $\infty$ and  a constant $c_3>0$ independent of $l$, such that
\[\mathop {\lim }\limits_{k \to \infty } \int\limits_{{\mathbb{S}^{n - 1}}} {h_{{P_{{j_{{l_k}}}}}}^{1 - p}d{S_{{P_{{j_{{l_k}}}}}}}}  \ge {c_3}\int\limits_0^r {{t^{n - p - 1}}dt} .\]
Hence,
\[\infty  > {\sup _l}\left\{ {|{\mu _{{j_l}}}|} \right\} \ge {c_2}{c_3}\int\limits_0^r {{t^{n - p - 1}}dt} ,\]
which implies that $p<n$.
\end{proof}

From Theorem \ref{generalCapMProblem}, we  immediately obtain the following results.

\begin{corollary}\label{generalCapMProblem2}
Suppose  $1< p<\infty$, $1<\mathfrak{p}\le 2$ and $
n-\mathfrak{p}\ne p$.
If $\mu$ is a finite Borel measure on $\mathbb{S}^{n-1}$
which is not concentrated on any closed hemisphere, then there exists a unique convex body $K$ in
$\mathbb{R}^n$ containing the origin, such that
\[h_K^{p-1}d\mu=d\mu_\mathfrak{p}(K,\cdot).\]
If in addition $p\ge n$, then $K\in\mathcal{K}^n_o$.
\end{corollary}

\begin{proof}
By Theorem \ref{generalCapMProblem}, there exists a unique convex
body $K^*$ containing the origin, such that
$C_{\mathfrak{p}}(K^*)h_{K^*}^{p - 1}d\mu
=d\mu_{\mathfrak{p}}(K^*,\cdot)$. Let $K = {\rm
C}_\mathfrak{p}(K^*)^{1/(p+\mathfrak{p}-n)}K^*$. Then, $h_K^{p -
1}d\mu = d{\mu _\mathfrak{p}}(K, \cdot )$.
\end{proof}

\begin{corollary}\label{generalCapMProblem3}
Suppose   $1< p<\infty$ and  $1<\mathfrak{p}\le 2$. If $\mu$ is a
finite even Borel measure on $\mathbb{S}^{n-1}$ which is not
concentrated on any great subsphere, then there exists a unique
origin-symmetric convex body $K$ in $\mathbb{R}^n$, such that $
{{\rm C}_\mathfrak{p}(K)}^{-1}{\mu_{p,\mathfrak{p}}(K,\cdot)}=\mu$.
\end{corollary}

\begin{proof}
By Theorem \ref{generalCapMProblem}, there exists a unique convex
body containing the origin, such that $h_K^{p-1}d\mu={\rm
C}_\mathfrak{p}(K)^{-1}d\mu_\mathfrak{p}(K,\cdot)$. Since $\mu$ is
even, it implies that $h_{-K}^{p-1}d\mu={\rm
C}_\mathfrak{p}(-K)^{-1}d\mu_\mathfrak{p}(-K,\cdot)$. So, the
uniqueness of $K$ in turn implies that $-K=K$.
\end{proof}

Consequently, if $ n-\mathfrak{p}\ne p$, then there exists a unique
origin-symmetric convex body $K'$ in $\mathbb{R}^n$, such that $
\mu= \mu_{p,\mathfrak{p}}(K',\cdot)$.

\vskip 25pt
\section{\bf  Continuity}
\vskip 10pt

Let $1<p<\infty$, $1<\mathfrak{p}\le 2$ and $\mathfrak{p}<n$.
Write $\mathcal{M}$ for the set of  finite Borel measures on
$\mathbb{S}^{n-1}$ which are not concentrated on any closed
hemisphere. For each $\mu\in\mathcal{M}$, denote by ${\rm
C}_\mathfrak{p}^p\mu$ the unique solution to Problem 5 for $(\mu,
\mathfrak{p},p)$, i.e., the unique convex body containing the origin
such that
\[\frac{d\mu_\mathfrak{p}(K,\cdot)}{{\rm C}_\mathfrak{p}(K)}=h_K^{p-1}d\mu.\]

A natural question about the continuity of solution to $L_p$
Minkowski problem for $\mathfrak{p}$-capacity asks the following:
\emph{If $\{\mu_j\}_j\subset \mathcal{M}$ converges to
$\mu\in\mathcal{M}$ weakly, is this the case that ${\rm
C}_\mathfrak{p}^p\mu_j\to {\rm C}_\mathfrak{p}^p\mu$?}

We answer this question affirmatively.

\begin{theorem}\label{Theorem9.1}
Suppose that $\mu_j,\mu\in\mathcal{M}$, $j\in\mathbb{N}$, $1<p<\infty$ and  $1<\mathfrak{p}\le 2$.
If $\mu_j\to\mu$ weakly  as
$j\to\infty$,  then
${\rm C}_\mathfrak{p}^p\mu_j\to {\rm C}_\mathfrak{p}^p\mu$.
\end{theorem}

\begin{proof}
For the sake of simplicity, write $K_j$ and $K$ for ${\rm
C}_\mathfrak{p}^p\mu_j$ and ${\rm C}_\mathfrak{p}^p\mu$,
respectively. By Lemma \ref{lemma6.3},  $K_j$ and
$K$ are also the unique solutions to Problem 3 for $\mu_j$ and $\mu$,
respectively. From Lemma \ref{lemma7.1}, it follows that the sequence $\{K_j\}_j$ is
bounded from above. Hence, to prove that $K_j\to K$,  it suffices to
prove each convergent subsequence $\{K_{j_l}\}_l$ of
$\{K_j\}_j$ converges to $K$.

Assume that $\{K_{j_l}\}_l$ is a convergent subsequence of $\{K_j\}_j$.
Let $\bar{K}_j={\rm
C}_\mathfrak{p}(K_j)^{-1/(n-\mathfrak{p})}K_j$,  $j\in\mathbb{N}$. By Lemma
\ref{lem6.1} (1),   $\bar{K_{j_l}}$ is the unique
solution to Problem 4 for $\mu_{j_l}$. From Lemma \ref{lemma7.1}
again, the sequence $\{\bar{K}_{j_l}\}_l$ is bounded from above.
 Thus, by the Blaschke
selection theorem, $\{\bar{K}_{j_l}\}_l$ has a subsequence
$\{\bar{K}_{j_{l_i}}\}_i$ converging to a compact convex set
$\bar{K}_0$. By Lemma \ref{lemma7.4} (1),  $\bar{K}_0$ is a convex
body containing the origin; by Lemma \ref{lemma7.4} (2),
$0<\int_{\mathbb{S}^{n-1}}h_{\bar{K}_0}^pd\mu<\infty$. Thus,
\[K_0=\left(\frac{\mathfrak{p}-1}{n-\mathfrak{p}}\int\limits_{\mathbb{S}^{n-1}}h_{\bar{K}_0}^pd\mu \right)^{\frac{-1}{p}}\bar{K}_0\]
is indeed a convex body. By Lemma \ref{lemma7.4} (3), $K_0$ is the unique solution to
Problem 5 for $\mu$. In light of $K$ is also the unique solution to Problem 5 for $\mu$,  we have
$ K_0=K. $
Therefore, $\lim_{i\to\infty}K_{j_{l_i}}=K$. Since
$\{K_{j_l}\}_l$ is a convergent sequence, it follows that
$\lim_{l\to\infty} K_{j_l}=K$.
\end{proof}

For each $\mu\in\mathcal{M}$, if $n-\mathfrak{p}\neq p$, we can define
\[\bar{\rm C}_\mathfrak{p}^p\mu ={\rm C}_\mathfrak{p}({\rm C}_\mathfrak{p}^p\mu)^{-1/(n-\mathfrak{p}-p)}{\rm C}_\mathfrak{p}^p\mu. \]
Then $\bar{\rm C}_\mathfrak{p}^p\mu$ is the unique convex
body which contains the origin and is such that
\[ h_{\bar{\rm C}_\mathfrak{p}^p\mu}^{p-1}d\mu=d\mu_\mathfrak{p}(\bar{\rm C}_\mathfrak{p}^p\mu,\cdot). \]

\begin{corollary}\label{Corollary9.2}
Suppose that $\mu_j,\mu\in\mathcal{M}$,
$j\in\mathbb{N}$, $1<p<\infty$ and $1<\mathfrak{p}\le 2$,
$n-\mathfrak{p}\neq p$.  If $\mu_j\to \mu$ weakly as $j\to\infty$, then
$\bar{\rm C}_\mathfrak{p}^p\mu_j\to \bar{\rm C}_\mathfrak{p}^p\mu$.
\end{corollary}

\begin{proof}
Since $\mu_j\to \mu$ weakly, we have  $ {\rm
C}_\mathfrak{p}^p\mu_j\to  {\rm C}_\mathfrak{p}^p\mu$ by Theorem
\ref{Theorem9.1}. So, ${\rm C}_\mathfrak{p}({\rm
C}_\mathfrak{p}^p\mu_j)\to {\rm C}_\mathfrak{p}({\rm
C}_\mathfrak{p}^p\mu)$, and therefore ${\rm C}_\mathfrak{p}({\rm
C}_\mathfrak{p}^p\mu_j)^{-1/(n-\mathfrak{p}-p)}\to {\rm
C}_\mathfrak{p}({\rm C}_\mathfrak{p}^p\mu)^{-1/(n-\mathfrak{p}-p)}$,
as $j\to\infty$. Consequently,
\[\lim_{j\to\infty}\bar{\rm C}_\mathfrak{p}^p\mu_j
=\lim_{j\to\infty}{\rm C}_\mathfrak{p}({\rm
C}_\mathfrak{p}^p\mu_j)^{-1/(n-\mathfrak{p}-p)}{\rm
C}_\mathfrak{p}^p\mu_j ={\rm C}_\mathfrak{p}({\rm
C}_\mathfrak{p}^p\mu)^{-1/(n-\mathfrak{p}-p)}{\rm
C}_\mathfrak{p}^p\mu= \bar{\rm C}_\mathfrak{p}^p\mu,\]
as desired.
\end{proof}

\begin{corollary}\label{Corollary9.3}
Suppose that $K_j,K\in\mathcal{K}^n_o$,
$j\in\mathbb{N}$, $1<p<\infty$ and  $1<\mathfrak{p}\le 2$,
 $n-\mathfrak{p}\neq p$. If
$\mu_{p,\mathfrak{p}}(K_j,\cdot)\to\mu_{p,\mathfrak{p}}(K,\cdot)$
weakly as $j\to\infty$,  then $K_j\to K$.
\end{corollary}

\begin{proof}
Let $\mu_j=\mu_{p,\mathfrak{p}}(K_j,\cdot)$ and
$\mu=\mu_{p,\mathfrak p}(K,\cdot)$. Then,
$h_{K_j}^{p-1}d\mu_j=d\mu_\mathfrak{p}(K_j,\cdot)$, and
$h_{K}^{p-1}d\mu_j=d\mu_\mathfrak{p}(K,\cdot)$. From the uniqueness
of $\bar{\rm C}_\mathfrak{p}^p$ it follows that $K_j=\bar{\rm
C}_\mathfrak{p}^p\mu_j$ and $K=\bar{\rm C}_\mathfrak{p}^p\mu$. Since
$\mu_j\to\mu$ weakly, Corollary \ref{Corollary9.2} implies that
$\bar{\rm C}_\mathfrak{p}^p\mu_j\to \bar{\rm C}_\mathfrak{p}^p\mu$,
as $j\to\infty$. That is, $K_j\to K$ as $j\to\infty$.
\end{proof}

\vskip 10pt\noindent \textbf{Remark 9.4.}  After this work, we further study the $L_{p}$ Minkowski problem for $\mathfrak{p}$-capacity
when the given measure is even, it will be dealt with in a separate paper as a sequel.

\vskip 25pt
\section{\bf  Open problem}
\vskip 10pt
Since  the logarithmic Minkowski problem  is the most important case, we pose the following

\vskip 10pt \noindent \textbf{Logarithmic Minkowski problem for capacity.} \emph{Suppose that $\mu$ is a finite Borel measure 
on $\mathbb{S}^{n-1}$ and $1<\mathfrak{p}<n$. What are the 
necessary and sufficient conditions on $\mu$ so that  $\mu$ is the $L_{0}$ $\mathfrak{p}$-capacitary measure  $\mu_{0,\mathfrak{p}}(K,\cdot)$
of a convex body $K$ in $\mathbb{R}^n$?}

\vskip 20pt
\bibliographystyle{amsplain}

\end{document}